\documentclass[oneside, 11pt]{amsart}

\usepackage{amssymb, amsmath}

\usepackage{tikz, tikz-cd}
\usetikzlibrary{calc, decorations.markings,intersections,through, math}
\tikzstyle{bbullet}=[circle, fill=black, inner sep=0pt, minimum size=3pt]
\tikzstyle{gbullet}=[circle, fill=green, inner sep=0pt, minimum size=3pt]
\tikzstyle{rbullet}=[circle, fill=red, inner sep=0pt, minimum size=3pt]
\tikzstyle{wbullet}=[circle, fill=white, thick, inner sep=0pt, minimum size=3pt]
\tikzstyle{cross}=[circle, draw=black, inner sep=0pt, minimum size=3pt]

\usepackage{enumitem}
\usepackage{mathtools}

 \usepackage[textwidth=125mm, textheight=185mm]{geometry}

\newcommand{\me}{\mathcal{E}}
\newcommand{\ml}{\mathcal{L}}
\newcommand{\mo}{\mathcal{O}}
\newcommand{\mt}{\mathcal{T}}
\newcommand{\mx}{\mathcal{X}}

\newcommand{\CC}{\mathbb{C}}
\newcommand{\GG}{\mathbb{G}}
\newcommand{\PP}{\mathbb{P}}
\newcommand{\QQ}{\mathbb{Q}}
\newcommand{\RR}{\mathbb{R}}
\newcommand{\ZZ}{\mathbb{Z}}

\newcommand{\Aff}{\mathrm{Aff}}
\newcommand{\Aut}{\mathrm{Aut}}
\newcommand{\Bir}{\mathrm{Bir}}
\newcommand{\Bl}{\mathrm{Bl}}
\newcommand{\Diff}{\mathrm{Diff}}
\newcommand{\GL}{\mathrm{GL}}
\newcommand{\id}{\mathrm{id}}

\newcommand{\Ker}{\mathrm{Ker}}

\newcommand{\pr}{\mathrm{pr}}

\newcommand{\Stab}{\mathrm{Stab}}

\newcommand{\topo}{\mathrm{top}}

\theoremstyle{plain}
 \newtheorem{thm}{Theorem}[section]
 \newtheorem{theo}[thm]{Theorem}

\newtheorem{cor}[thm]{Corollary}

\theoremstyle{definition}

\newtheorem{defin}[thm]{Definition}

\newtheorem{example}[thm]{Example}

\newtheorem{lem}[thm]{Lemma}

\newtheorem{principle}[]{Principle}

\newtheorem{question}[thm]{Question}

\newtheorem{rmk}[thm]{Remark}

\newcommand{\sE}{{\mathcal E}}

\newcommand{\sG}{{\mathcal G}}

\newcommand{\sK}{{\mathcal K}}

\newcommand{\sS}{{\mathcal S}}

\newcommand{\sX}{{\mathcal X}}

\newcommand{\C}{{\mathbb C}}

\newcommand{\E}{{\mathbb E}}
\newcommand{\F}{{\mathbb F}}

\newcommand{\hol}{\ensuremath{\mathcal{O}}}

\newcommand\om{\omega}
\newcommand\la{\lambda}

\newcommand\s{\sigma}

\newcommand\al{\alpha}
\newcommand\be{\beta}
\newcommand\Ga{\Gamma}
\newcommand\De{\Delta}
\newcommand\ga{\gamma}
\newcommand\de{\delta}

\DeclareMathOperator{\Pic}{Pic}

\DeclareMathOperator{\Hom}{Hom}
\DeclareMathOperator{\Alb}{Alb}

\newcommand{\FF}{\ensuremath{\mathbb{F}}}

\newcommand{\ra}{\ensuremath{\rightarrow}}

\newcommand\dual{\mathrel{\raise3pt\hbox{$\underline{\mathrm{\thinspace d
\thinspace}}$}}}
\newcommand\qe{\ifhmode\unskip\nobreak\fi\quad $\Box$}       % box for QED

\def\BOX{\hfill\lower.5\baselineskip\hbox{$\Box$}}
\def\C{{\Bbb C}}

\newenvironment{dedication}
        {\begin{quotation}\begin{center}\begin{em}}
        {\par\end{em}\end{center}\end{quotation}}

\begin{document}
\title[Topologically trivial automorphisms of cKM]{On topologically trivial automorphisms of compact K\"ahler manifolds and algebraic  surfaces}
\author{Fabrizio Catanese}
\author{Wenfei Liu}

\address{Lehrstuhl Mathematik VIII, Mathematisches Institut der Universit\"{a}t
Bayreuth, NW II, Universit\"{a}tsstr. 30,
95447 Bayreuth, Germany,  and Korea Institute for Advanced Study, Hoegiro 87, Seoul, 
133-722, Korea}
\email{Fabrizio.Catanese@uni-bayreuth.de}

\address{Xiamen University  \\ School of Mathematical Sciences \\ Siming South Road 422 \\ Xiamen, Fujian 361005 (China)}
\email{wliu@xmu.edu.cn}

\thanks{The first author acknowledges support of the ERC 2013 Advanced Research Grant-340258-TADMICAMT; part of this work was performed at KIAS Seoul.
 The second author acknowledges support of  the NSFC (No.~11971399 and No.~11771294). } 
\begin{abstract}
In this paper, we investigate automorphisms of compact K\"ahler manifolds with different levels of topological triviality. In particular, we provide several examples of smooth complex projective surfaces $X$ whose  groups of $C^\infty$-isotopically trivial automorphisms, resp.~
cohomologically trivial  automorphisms, have a number of connected components  which can be arbitrarily large.
\end{abstract}
\maketitle
\begin{dedication}
Dedicated to the memory of the `red' Bishop of Italian Mathematics, Edoardo Vesentini  (1928--2020).
\end{dedication}

\tableofcontents

\section{Introduction}
 Let $X$ be a compact connected complex manifold. Bochner and Montgomery \cite{bm1, bm2} showed that the automorphism group $\Aut(X)$
(the group of biholomorphic maps
$ g\colon X \ra X$, i.e., the group of  diffeomorphisms  
$ g \in \Diff (X)$
which preserve the complex structure of $X$) 
is a finite dimensional complex Lie Group, possibly with infinitely many connected components, whose Lie Algebra is the space
$H^0(X,  \Theta_X)$ of holomorphic vector fields on $X$.

 Denote by $\Aut_0(X)$ the identity component of $\Aut(X)$ and define the group of  
 $C^\infty$-isotopically
 trivial automorphisms as:
$$\Aut_*(X) : = \{\sigma\in\Aut(X) \mid \sigma \in \Diff_0(X)\},$$
where $\Diff_0(X)$ denotes the identity component of the  group of diffeomorphisms.

 In other words, $\Aut_*(X)$ consists of the automorphisms that are $C^\infty$-isotopic to the identity.
  This group plays an important role (\cite{handbook}) in the construction of the Teichm\"uller space of $X$,
 and   Meersseman recently constructed the Teichm\"uller stack $\mt(X)$ of complex structures on the underlying differentiable manifold of $X$ (\cite{Me19}). 
 
 The holonomy 
of $\mt(X)$ turns out to be the quotient group   $\Ga_{*}(X) : = \Aut_{*}(X)/\Aut_0(X)$,
 a subgroup of the group of (connected) components $\Ga(X) : = \Aut(X)/\Aut_0(X)$.

 The group of (connected) components $\Ga(X) : = \Aut(X) / \Aut_0(X)$ is at most countable, and here is an easy example where
 it is infinite:

\begin{example}
Let $E$ be an elliptic curve, and let $X = E^n$ with $n\geq 2$. 

Then $\Aut_0(X) = E^n$, while the group 
 $\Ga(X)$ contains $ \GL(n, \ZZ)$, acting in the obvious way:
 $$ g \in  \GL(n, \ZZ) , x = (x_1, \dots, x_n) \mapsto gx= (\sum_j g_{1j} x_j, \dots , \sum_j g_{nj} x_j).$$
\end{example}

Since  the condition for two automorphisms to be isotopic is not directly tractable by algebro-geometrical methods, the strategy is to first consider the action of $\Aut(X)$ on the cohomology groups $H^*(X;R)$, where $R$ is a coefficient ring.  We denote by
\[
\Aut_R(X):=\{\sigma\in\Aut(X) \mid \sigma \text{ induces the trivial action on }H^*(X;R)\}
\]
In practice, we choose $R=\ZZ, \QQ, \RR$ or $\CC$. One more equivalence relation among  automorphisms is the homotopy equivalence, so we define the group of   
 homotopically
trivial automorphisms as:
\[
\Aut_\sharp(X) = \{\sigma\in\Aut(X) \mid \sigma \text{ is homotopic to }\id_X\},
\]
It is clear that 
\begin{multline*}
   \Aut_0(X) \vartriangleleft \Aut_*(X) \vartriangleleft \Aut_\sharp(X)\vartriangleleft \Aut_\ZZ(X)\\
   \vartriangleleft \Aut_\QQ(X) = \Aut_\RR(X) =\Aut_\CC(X) \vartriangleleft \Aut(X)
\end{multline*}
 so that it suffices to consider the smaller ladder
$$
   \Aut_0(X) \vartriangleleft \Aut_*(X) \vartriangleleft \Aut_\sharp(X)\vartriangleleft \Aut_\ZZ(X)
   \vartriangleleft \Aut_\QQ(X)  \vartriangleleft \Aut(X).
$$

The case where $X$ is a cKM = compact K\"ahler Manifold, with a K\"ahler metric $\om$,  was
considered around 1978 by Lieberman \cite{Li78} and Fujiki \cite{fujiki}, in particular Lieberman \cite{Li78} proved: 
\begin{theo}[Lieberman]\label{thm: Lieberman}
$\Aut_0(X)$ is a finite index  sugbroup of the group of automorphisms preserving the cohomology class of the K\"ahler form, 
$$\Aut_{\om}(X) = \{\sigma\in\Aut(X) \mid \sigma^* [\om] = [\om]\}.$$

In particular, the quotient group  $$ \Ga_{\QQ}(X) : = \Aut_\QQ(X) / \Aut_0(X) $$ is a  finite group.

\end{theo}

For complex dimension $n=1$, it is well known that everything simplifies,  in fact for $n=1$ $\Aut_0(X) = \Aut_\QQ(X)$.
But already for $ n  = 2$ the situation is extremely delicate, hence  this paper is dedicated to the case of complex dimension $n=2$,
which provides examples where the different groups of components can have arbitrarily high cardinality.

For surfaces of general type, 
essentially by the Bogomolov--Miyaoka--Yau inequality (final result in \cite{miyaoka}), there is a constant $C$ such that $|\Aut_\QQ(X)|<C$ for any surface of general type (see \cite{Cai04}).  For  surfaces of general type the important open question is whether they are rigidified in the sense of 
\cite{handbook}, that is, $\Aut_*(X)$ is a trivial group (see the work of Cai-Liu-Zhang   \cite{clz} and of  the second author  \cite{CL18}, \cite{Liu18} for recent results in  the study of
the group $\Aut_\QQ(X)$).

For surfaces not of general type the aim is to describe the group 
$ \Ga_{\QQ}(X) : = \Aut_\QQ(X) / \Aut_0(X) $.

In 1975 Burns and Rapoport \cite{br} proved that, for a K3 surface $X$, $\Aut_\QQ(X) $ is a trivial group.
 Peters \cite{Pe79}, \cite{Pe80} began the study of $\Aut_\QQ(X) $ for compact K\"ahler surfaces. 
Automorphisms of surfaces were also investigated by Ueno \cite{ueno} and  Maruyama \cite{Ma71}
in the 70's, then by Mukai and Namikawa \cite{MN84}.

The main results of  this paper can be summarized in  the following main theorem, which is obtained from several more
precise theorems.

\medskip

\noindent {\bf Main Theorem.} {\em 
The indices $[\Aut_\QQ(X):\Aut_\ZZ(X)],\, [\Aut_\ZZ:\Aut_*(X)]$, and $[\Aut_*(X):\Aut_0(X)]$  can be arbitrarily large for smooth projective surfaces not of general type.}

\medskip

The first result which we prove (contradicting earlier assertions of other authors)  answers questions raised by Meersseman \cite{Me17} and Oguiso \cite{O20}:

\medskip

\noindent{\bf Theorem \ref{rational}.}
For each positive integer $m$ there exists a rational surface $X$, blow up of $\PP^2$, such that $\Aut_{\QQ}(X) = \Aut_*(X) \cong \ZZ/ m\ZZ$.

 A similar  example can be constructed for other non-minimal uniruled surfaces: just blowing up appropriately  any smooth projective surface with an (effective) $\CC^*$-action. 
 
 Surprisingly,  the same unboundedness phenomenon for $$\Ga_*(X): = \Aut_*(X) / \Aut_0(X)$$
  can happen also for minimal ruled surfaces:
   
 \begin{thm}
 Let $E$ be an elliptic curve and let $X : = \PP (\hol_E \oplus \hol_E(D))$ where $D$ is a divisor of even positive degree $d=2m$.
 
 Then $\Ga_*(X)$ surjects onto $(\ZZ/m\ZZ)^2 $.
 \end{thm}
  
In the rest of the paper we consider more   general results concerning the various subgroups of the ladder in terms of the
 Enriques-Kodaira  classification of 
  compact
 K\"ahler surfaces.
 
 We first consider rational surfaces which are blow-ups of $\PP^2$,  using 
 Principle~\ref{prin: descend} of Section~\ref{sec: prelim}
 and then we pass to  consider minimal surfaces,
 starting from $\PP^1$-bundles over curves.
 
 We describe then the  situation for 
  non ruled surfaces: 
 the case of Kodaira dimension 
  $\kappa(X)=0$
  is pretty well understood, down here  a summary of the  results.

 \begin{itemize}[leftmargin=*]
  \item  $\Aut_\#(X) = \Aut_0(X)$ holds for each surface $X$ with $\kappa (X)=0.$ 
  \item
  For complex tori and their blow ups $X$, $ \Aut_0(X) = \Aut_\QQ(X)$.
 \item
 For K3 surfaces (hence on their blow ups) $\Aut_\QQ(X) = 
 \{ \id_X \}.
 $
 \item
 For 
   Enriques surfaces  Mukai and Namikawa \cite{MN84} proved that 
$|\Aut_\QQ(X)| \leq 4$, and  that  there are examples with $\Aut_\ZZ(X)= \ZZ/2\ZZ$.  
\item
For 
  hyperelliptic surfaces $X$, we show that  $\Aut_\ZZ(X)= \Aut_0(X) $ is isogenous to $ \Alb(X)$, 
while the group $ \Ga_{\QQ} = \Aut_\QQ(X) / \Aut_\ZZ(X)$
can be described in each case.

$ \Ga_{\QQ}$ is a group of order $\leq 12$, and the case of order $12$ occurs precisely with the alternating group $\mathfrak A_4$.
 \end{itemize}

 We postpone to the sequel to this  paper the full treatment of the more delicate   case 
 where the Kodaira dimension  $\kappa(X)=1$.
 
 In this paper  we use examples of surfaces  in this class in order to
 show unboundedness also for other quotients of the  ladder 
$$
   \Aut_0(X) \vartriangleleft \Aut_*(X) \vartriangleleft \Aut_\sharp(X)\vartriangleleft \Aut_\ZZ(X)
   \vartriangleleft \Aut_\QQ(X).
$$

Combining theorems \ref{Kod1min} and \ref{Kod1nonmin} we get:
\begin{theo}
i) For each positive integer $n$ there exists a  minimal surface $X$
 of Kodaira dimension $1$ such that 
$
[\Aut_\QQ(X):\Aut_\ZZ(X)] \geq n.$

ii) For each positive integer $n$ there exists a  (non minimal) surface $X$
 of Kodaira dimension $1$ such that $\Aut_*(X)=\{\id_{X}\}$, and  
$$
\Aut_\ZZ(X)= \ZZ/n\ZZ.
$$
\end{theo}

The difference between $\Aut_\sharp(X)$ and   $\Aut_*(X)$ is still mysterious, even if we show that
in many cases the two groups coincide. We pose the following

\noindent {\bf Question 1.} Is $[\Aut_\sharp(X) : \Aut_*(X)]$ uniformly  bounded?

\section{Elementary observations and basic principles}\label{sec: prelim}
\begin{principle}\label{prin: negative}
Let $ \sigma \in \Aut_\QQ(X )$, where $X$ is a (compact complex connected) surface, and let $C$ be an irreducible curve with $C^2 < 0$.

Then $\sigma(C) = C$.
\end{principle}
\begin{proof} Assume in fact that the irreducible curve $ \sigma (C)$ is different from  $C$: then  $  C \cdot \sigma(C) \geq 0$.

But since $\sigma(C)$ has the same rational cohomology class of $C$, we have  $C \cdot \sigma(C) = C^2 < 0$, a contradiction.
\end{proof}

\begin{principle}\label{prin: red fibre}
Let $ f : X \ra B$ be a fibration of the surface $X$ onto a curve, and $ \sigma \in \Aut_\QQ(X )$: 
then $\sigma$ preserves the fibration, that is, there is an action   $\bar{\sigma}$ on $B$ such that $\bar{\sigma} \circ f = f \circ \sigma$.

If the genus of $B$ is at least $1$, then the action of $\bar{\s}$ on $B$ is trivial, unless if  $B$ has genus  $1$
and $\bar{\s}$ has no fixpoints on $B$. 

Moreover, if $F''$ is a reducible fibre, then   $\sigma(F'') = F''$.
\end{principle}

 For simplicity, and because of equivariance ($\bar{\sigma} \circ f = f \circ \sigma$),  by a small abuse of notation
 we shall use the notation
$\s$ also for $\bar{\s}$.

\begin{proof} Let $F$ be an irreducible fibre of $f$. Then $\sigma(F) \cdot F = 0$, hence $\sigma(F)$ is contained in another
fibre of $f$; since $\sigma(F)$ is  irreducible, it is another fibre of $f$.

Since $ \sigma \in \Aut_\QQ(X )$, and $H^1(B, \QQ) \subset H^1 (X, \QQ)$, by Lefschetz's principle $\s$
acts trivially on $B$ if  the genus of $B$ is at least $2$, and is a translation if the genus of $B$ is $=1$.

The last  assertion follows  from Principle~\ref{prin: negative} and Zariski's lemma (the components of $F''$ have negative self-intersection).
\end{proof}

\begin{principle}\label{prin: multiplefibres} 
Let $ f : X \ra B$ be a fibration of the surface $X$ onto a curve, $ \sigma \in \Aut_\ZZ(X )$,
and let $ F'' = m F'$ be  a multiple fibre of $f$ with $F'$ irreducible. Then $\sigma(F'') = F''$, unless possibly if 
the genus $g$ of $B$ is $\leq 1$, $m=2$,
there are only two multiple fibres  with multiplicity $2$, they are isomorphic to each other, and all the other
multiple fibres  have
odd multiplicity.
 \end{principle}

\begin{proof}
By Principle \ref{prin: red fibre},  $\sigma$ acts on the fibration, in particular fixing the reducible fibres, and permuting the multiple fibres
having the same multiplicity.  Let $g$ be the genus of $B$. We saw in Principle \ref{prin: red fibre} that $\sigma$ acts as the identity on $B$ if $g \geq 2$.

We have \cite{barlotti}, \cite{cime} the exact sequence of the orbifold fundamental group of the fibration $f$:
$$ \pi_1 (F ) \ra \pi_1(X) \ra \pi_1^{orb} (f) \ra 1,$$
where $F$ is a smooth fibre, and, letting $\{P_1, \dots, P_r\}$ be the set of points whose inverse images are the multiple fibres
of $f$, $f^{-1}(P_j) = m_j F'_j$, then 
$$\pi_1^{orb} (f) : = \langle \al_1, \dots, \al_g, \be_1, \dots, \be_g, c_1, \dots c_r \mid \Pi_1^g [\al_i, \be_i] c_1\cdot  \dots \cdot  c_r=1,
c_j^{m_j} =1\rangle.$$
 $\s$ acts on the fibration, hence, choosing a path from the base point $x_0$ to $\s(x_0)$, we get an action of
$\s$ on $\pi_1(X)$, which is  defined not uniquely, but  only up to an inner automorphism, and which leaves invariant the normal subgroup
which is the image of $ \pi_1 (F )$, since $\s$ sends a smooth fibre to another smooth fibre. Hence with this choice of such a path, we get an action of $\s$ on $ \pi_1^{orb} (f)$.

 If we pass to the respective Abelianizations, we get a uniquely  defined action.

Hence we have a surjection
$$ H_1(X, \ZZ) \ra ( \oplus_1^g \ZZ \al_j  \oplus_1^g \ZZ \be_j  \oplus_1^r (\ZZ / m_i \ZZ) c_i)  / \langle (\sum_i c_i)\rangle$$
on which $\s$ acts equivariantly.

Since $\sigma$ is now assumed to act  trivially on homology, it acts trivially on the quotient group. Assume that $F''$ corresponds to the point 
$P_1$.
Clearly $\sigma$ can only send $F''$ to a multiple fibre of the same multiplicity, and isomorphic to $F''$.
Assume that there is such a fibre, and that it corresponds to the point $P_2$. There remains to see whether $c_1$ and $c_2$
can have the same image in the quotient group.

This means that the vector $e_1 - e_2\in \oplus_1^r \ZZ e_i
$ is an integral linear combination of the  relation vectors $m_1 e_1,\, \dots ,\, m_r e_r,\, e : = \sum_1^r e_i$:
\[
e_1-e_2=a_1m_1e_1+\dots + a_rm_r e_r + ae
\]
where $a_1, \dots, a_r$ and $a$ are integers.

Since $m_1 = m_2$, this implies that $1-a , -1 -a$ are divisible by $m_1$,
hence $ 2 $ is divisible by $m_1$, hence $m_1=2$. Since $a$ is then odd, and $m_j$ divides $a$ for all $ j \geq 3$,
this is possible if and only if all the other $m_j$ for $ j \geq 3$ are odd, and then one can take $a$ as the least common multiple of
$m_3, \dots, m_r$.
\end{proof}

The next Principle~\ref{SIPU} is a special case of a more general result, and is
based on the technique of surfaces isogenous to a product and of unmixed type (\cite{isogenous}).

\begin{defin}
A surface  $X$
 is said to be  a SIP = Surface Isogenous to a Product if  $X$ is the quotient
of a product of curves of genus $\geq 1$ by the free action of a finite group $G$:
$$
X
 = (C_1 \times C_2)/G.$$
We speak of a higher product if both curves $C_1, C_2$ have genus at least $2$.

The action of $G$ is said to be UNMIXED if  $G$ acts on each factor $C_i$, and diagonally on the product,
$ \sigma(x,y) : = (\sigma x, \sigma y),$ so that more precisely 
$
X
 = (C_1 \times C_2)/\De_G,$
where $\De_G \subset G \times G$ is the diagonal subgroup and $G \times G$ acts on $C_1 \times C_2$.

We can assume that we have a minimal realization, that is, $G$ acts faithfully on each  factor $C_i$.

Observe that 
$\Aut(S)$
 contains the quotient $N_{\De_G}/ \De_G$, where $N_{\De_G}$ is the normalizer of $\De_G$
inside $ \Aut(C_1) \times \Aut(C_2)$ (and the two groups are  equal if we have a higher  product with $C_1, C_2$ not isomorphic).

\end{defin}

 \begin{principle}\label{SIPU} Let 
$X$
  be  a SIP of unmixed type, that is, a surface isogenous to a product of unmixed type,
 with a minimal realization 
 $
X
  = (C_1 \times C_2)/\De_G.$ 

Then 
$\Aut_\QQ(X )$
is the subgroup of  $N_{\De_G}/ \De_G$ corresponding to automorphisms 
$h(x,y) =  (h_1(x), h_2(y))$ acting trivially on 
$$ H^* (C_1 \times C_2, \QQ)^{\De_G}.$$
In particular $h_i$ acts trivially on $ H^1 (C_i, \QQ)^{G} = H^1(C_i /G, \QQ)$,
and  

(I) if $C_2 =:E$ has genus $1$ and $G$ acts freely on it, then $h_2$ is a translation;

(II) if $C_1$ has genus $\geq 2$ and $G$ acts freely on it, then we may represent an element
in  
$\Aut_\QQ(X)$
 by such an automorphism $h$ with $h_1 = \id_{C_1}$
and $h_2 \in Z_G$, where $Z_G$ is the centralizer of $G$ inside $\Aut(C_2)$.

(III) Assume that  $C_2 =: E$ has genus $1$,  and $G$ acts freely on $E$. Assume  moreover that 
if $C_1$ has  genus $g_1 \geq 2$, then  the orders of the stabilizers of points of $C_1$ are not of the form: $(2,\, 2,\, m_3,\, \dots)$, where the $m_i$ 's  are odd numbers.

  Then all automorphisms in 
$ \Aut_\ZZ(X)$
 are represented by pairs 
  $(h_1, h_2)\in \Aut(C_1) \times \Aut(E)$
  with $h_1 = \id_{C_1}$ and $h_2$ a translation. 
 In this case 
$ \Aut_\ZZ(X) =  \Aut_0(X) \cong E$.
 \end{principle}
\begin{proof}
For each $i=1,2$ we have a fibration $f_i :  X \ra C_i /G$, and, by Principle 3, 
$H : =  \Aut_\QQ(X)$
acts equivariantly on 
$X$
 and $C_i /G$. 

We have \cite{barlotti}, \cite{cime} the orbifold fundamental group of the fibration $f_i$:
$$ 1 \ra \pi_1 (C_j) \ra 
\pi_1(X)
\ra \pi_1^{orb} (f_i) \ra 1,$$
on which $H$ acts (here $\{i,j\} = \{1,2\}$).

Hence the elements of $H$ preserve the characteristic subgroup $\pi_1 (C_1) \times \pi_1 (C_2) $
of $\pi_1(X)$
and lift to $C_1 \times C_2$, preserving the horizontal and vertical leaves. 
Therefore these are represented by automorphisms in $ \Aut(C_1) \times \Aut(C_2)$.
Since such lift $h =(h_1, h_2)$ induces an action on  $(C_1 \times C_2)/\De_G$
we see that $h \in N_{\De_G}$ and $H$ is then a subgroup of $N_{\De_G}/ \De_G$.

Observing that $$ 
H^* (X, \QQ)
 \cong H^* (C_1 \times C_2, \QQ)^{\De_G},$$
we obtain the first assertion.

The second follows since 
$$ 
H^1 (X, \QQ)
\cong H^1 (C_1, \QQ)^G \oplus  H^1 (C_2, \QQ)^G = H^1 (C_1/G, \QQ) \oplus  H^1 (C_2/G, \QQ).$$

(I): then $H^1 (C_2, \QQ) =   H^1 (C_2/G, \QQ)$ and since $h_2$ acts trivially on it, it is a translation.

(II): then $C_1 / G$ has genus $\geq 2$, hence $h_1$ acts trivially on $C_1 / G$, hence $h_1 \in G$.
Multiplying by an element in $\De_G$, we may assume that $h_1 = \id_{C_1}$. Then the condition that
$(\id_{C_1}, h_2) \in N_{\De_G}$ is equivalent to $h_2 \in Z_G$.

(III): since  $ h \in 
\Aut_\ZZ(X)$, 
by Principle \ref{prin: multiplefibres}, $h$ acts on the fibration $f_1$ preserving its multiple fibres
unless possibly if  the genus $g'_1$ of $ C_1 / G$ is at most $1$ and   the multiplicities of the
multiple fibres  are  of the form: $(2,\, 2,\, m_3,\, \dots)$, where the $m_i$ 's  are odd numbers.

 We observe that in our situation these multiplicites are equal to the orders of the stabilizers of points of $C_1$: if $g_1 \geq 2$
 by assumption  these 
  are not of the form: $(2,\, 2,\, m_3,\, \dots)$, where the $m_i$'s  are odd numbers. If instead $g_1=1$,
  and these orders have this form, 
  then the quotient $ C_1 / G$ has genus $g'_1 = 0$, and Hurwitz' formula yields 
  $$ 2 = \frac{1}{2} +  \frac{1}{2}  + \sum_i ( 1 - \frac{1}{m_i}) ,$$
  which is manifestly impossible since, for $m_i$ odd, 
  $$ \frac{2}{3}  \leq ( 1 - \frac{1}{m_i}) < 1.$$ 

Therefore $h_1$
acts on $C_1 \ra C_1 /G$  fixing the branch points in $ C_1 /G$.  Since $h_1$ acts as the identity
on the cohomology of $ C_1 /G$, $h_1$ acts as the identity on $ C_1 /G$ if the
quotient has genus $g'_1 \geq 2$, or if $g'_1 =1$ and there is a branch point,
or if $g'_1 =0$ and there are at least $3$ branch points. 
 Since $C_1$ has genus $\geq 2$,
 one of the three possibilities must occur,
and $h_1 \in G$. Multiplying by an element in $\De_G$, we may assume that $h_1 = \id_{C_1}$:  (I)
shows that $h_2$ is a translation.  The group $\{(\id_{C_1}, h_2) \mid h_2 \in E \}  \cong E$ has an action which clearly descends to 
$X$.
Therefore we have shown that 
$ \Aut_\ZZ(X) =  \Aut_0(X) \cong E$.
\end{proof}
\medskip
Directly from
Principle~\ref{prin: negative} follows the next Principle~\ref{prin: descend}:

\begin{principle}\label{prin: descend}
Let $X$
be a compact complex surface, and let 
 $X=X_{n}\xrightarrow{f_n} X_{n-1}\xrightarrow{f_{n-1}} \cdots \xrightarrow{f_2}X_{1} \xrightarrow{f_1} X_{0}$
be a sequence of blow-downs of $(-1)$-curves. For $0\leq k\leq n-1$, let 
$P_k\in X_{k}$
 be the blown-up point. Then,   if 
$\Bir(X)$
 denotes the group of bimeromorphic self maps of 
$X$, then:
\begin{multline}\label{eq: descend Q}
\Aut_\QQ(X) = \{\sigma\in  \Aut_\QQ(X_0) \subset \Bir(X)\mid \text{ for any } 0\leq k\leq n-1, \\
  \sigma_k:=\sigma|_{X_{k}} \in \Aut_\QQ(X_{k})  
\text{ is such that } \sigma_k(P_{k})=P_{k}
 \}
\end{multline}
and
\begin{multline}\label{eq: descend Z}
\Aut_\ZZ(X) = \{\sigma\in \Aut_\ZZ(X_0) \subset \Bir(X)\mid \text{ for any } 0\leq k\leq n-1, \\
 \sigma_k:=\sigma|_{X_{k}} \in\Aut_\ZZ(X_{k})
 \text{ is such that }\sigma_k(P_{k})=P_{k}
 \}
 \end{multline}
\end{principle}
\begin{proof}
For $\sigma\in\Aut_\QQ(X)$, $\sigma$ acts trivially on  
\[
H^2(X,\QQ)= f_n^*H^2(X_{n-1}, \QQ)\oplus \QQ[E_n],
\]
where $E_n$ is the exceptional divisor of $f_n$, so it preserves $E_n$ by Principle~\ref{prin: negative}.  It follows that $\sigma$ descends to a cohomologically trivial automorphism of  $X_{n-1}$ preserving 
$P_{n-1}$.
 Conversely, any $\sigma\in \Aut_{\QQ}(X_{n-1})$ fixing $P_{n-1}$  lifts to a cohomologically  trivial automorphism of $X_{n}$.
 By induction on $k$, we obtain the equality \eqref{eq: descend Q}.

Exactly the same argument yields \eqref{eq: descend Z}.
\end{proof}

 We need the following rigidity of harmonic maps into Riemannian manifolds with nonpositive curvature. Note that holomorphic maps between K\"ahler manifolds are harmonic with respect to the K\"ahler metrics.
\begin{thm}\label{thm: Hartman}
Let $M$ and $N$ be compact Riemannian manifolds, such that $N$ has nonpositive sectional curvature.
If $\phi_0, \phi_1\colon  M \rightarrow N $ are homotopic harmonic maps, 
 then there is a $C^\infty$ homotopy $\Phi\colon M\times [0,1]\rightarrow N$ from $\phi_0$ to $\phi_1$ with the following properties:
\begin{enumerate}[leftmargin=*]
\item denoting $\phi_t(x):=\Phi(x,t)$ for $(x, t)\in M\times [0,1]$, the maps $\phi_t\colon M\rightarrow N$ are harmonic;
\item for fixed $x$, the arc $[0,1]\rightarrow N,\, t\mapsto \phi_t(x)$ is a geodesic arc with length independent of $x$, and $t$ proportional to the arc-length;
\item if for each $0\leq t\leq 1$, there is a point $x_t\in M$ such that $\phi_0$ and $\phi_t$ coincide at $x_t$, then $\Phi$ is the constant homotopy, that is, $\phi_t =\phi_0$ for any $0\leq t\leq 1$.
\end{enumerate}

\end{thm}

\begin{proof}
In view of \cite[(G)]{Har67}, only the last statement needs a proof. 

Let $d\colon N\times N \rightarrow \RR_{\geq 0}$ denote the distance function of the Riemannian manifold $N$. Since $N$ is compact, there is a constant $\delta>0$ such that for any $y_1, y_2\in N$ with $d(y_1,y_2)<\delta$ there is a unique minimizing geodesic joining $y_1$ and $y_2$. 

By the continuity of the homotopy $\Phi$, one can find a partition of the unit interval $0=t_0<t_1<\dots<t_k=1$, such that for each $0\leq i\leq k-1$
\[
|\phi_{t_i}, \phi_{t_{i+1}}|_\infty:=\max_{x\in M} d(\phi_{t_i}(x), \phi_{t_{i+1}}(x)) <\delta.
\]

We will show by induction on $i$ that $\phi_t=\phi_0$ for any $t\in [t_{i}, t_{i+1}]$. We can assume that the equality $\phi_{t_{i}}=\phi_0$ has been proven. By assumption, for each $0\leq i\leq k$, we have a point $x_{t_i}\in M$ such that $\phi_{t_i}(x_{t_i}) = \phi_0(x_{t_i})$. Since $|\phi_{t_i}, \phi_{t_{i+1}}|_\infty<\delta$, the arc $[t_i, t_{i+1}]\rightarrow N,\, t\mapsto \phi_t(x_{t_{i+1}})$ is the unique minimizing geodesic arc with the same starting and ending points, namely $\phi_0(x_{t_{i+1}}) = \phi_{t_{i+1}}(x_{t_{i+1}}) $; this means that it is the constant arc at $\phi_0(x_{t_{i+1}})$. By (2), we infer that $\phi_t=\phi_{t_i} = \phi_0$ for any $t\in [t_{i}, t_{i+1}]$; see also \cite[(F)]{Har67}. This finishes the induction step and hence the proof.
\end{proof}

\begin{principle}\label{prin: rigidity}
Let $X$ a compact K\"ahler manifold with topological Euler number $\chi_\topo (X)\neq 0$. Suppose that there is a generically finite proper holomorphic map $\rho\colon X\rightarrow Y$ onto a compact K\"ahler manifold $Y$ with nonpositive sectional curvature. Then 
\[
\Aut_\sharp(X)=\{\id_X\}.
\]
\end{principle}
\begin{proof} 
Let $\sigma\in\Aut_\sharp(X)$ be a homotopically trivial automorphism, with homotopy $\Sigma\colon X\times [0,1]\rightarrow X$ from 
 $\id_X$ to $\sigma$.
 By Theorem~\ref{thm: Hartman}, we can assume that $\sigma_t(x):=\Sigma(x,t)$ is harmonic in $x$ for each fixed $t\in[0,1]$. 

 For each $t\in[0,1]$, we have $[\Gamma_{\sigma_t}]= [\Delta_X]\in H^{2n}(X\times X)$, where $n=\dim X$, $\Gamma_{\sigma_t}$ denotes the graph of $\sigma_t$ and $\Delta_X\subset X\times X$ is the diagonal. Then 
\[
[\Delta_X]\cdot [\Gamma_{\sigma_t}] = [\Delta_X]^2 = \chi_\topo(X) \neq 0.
\]
and it follows that  $\Delta_X\cap \Gamma_{\sigma_t}\neq \emptyset$. In other words, there exists a point  
 $x_t\in X$
fixed by $\sigma_t$.

 The map $\Sigma$ gives a homotopy from 
 $\rho\circ\sigma_t$
to $\rho$
 for each $t\in [0,1]$.
 
 Since 
  $\rho\circ\sigma_t(x_t)=\rho(x_t)$,  
 we have $\rho\circ\sigma_t = \rho$ by Theorem~\ref{thm: Hartman}.  Since $\rho$ is generically finite, for a general point $y\in Y$, the inverse image $\rho^{-1}(y)$ is a finite set. It follows that $\sigma_t(x) =x$ for each $t\in[0,1], x\in \rho^{-1}(y)$. In other words $\sigma_t=\id_X$ for $t\in[0,1]$. In particular, 
  $\sigma = \sigma_1 = \id_X$. 
\end{proof}

\section{Unbounded $[\Aut_\QQ(X):\Aut_\ZZ(X)]$}
In this section, we construct a series of examples where $[\Aut_\QQ(X):\Aut_\ZZ(X)]$ is unbounded.

Recall that surfaces with  
 Kodaira dimension
 $\kappa(X)=1$ 
 are canonically elliptic, there is a fibration $f : X \ra B$   over a curve $B$ 
and with general fibre a smooth elliptic curve, such that $\Aut(X)$ acts equivariantly on $X, B$. 

\begin{theo}\label{Kod1min}
 For each positive integer $n$ there exists a  minimal surface $X$
 of Kodaira dimension $1$ such that 
$
[\Aut_\QQ(X):\Aut_\ZZ(X)] \geq n.$
\end{theo}

\begin{proof}
Let $B$ be the hyperelliptic curve, compactification of the affine curve of  equation:
\[
y^2 = x^n-1
\]
where $n = 2g(B) +2 \geq 6$ is an even integer. 
Let $F$ be an elliptic curve. The group $G=\langle\tau\rangle\cong\ZZ/2\ZZ$
acts on $B$ by
\[
\tau (x,y) =  (x,-y)
\]
and we make it act on $F$ by translations
\[
\tau(z) = z + \epsilon,\,\,\,\forall\, z
\in F
\]
where $\epsilon$ is  a torsion point of order precisely $2$.
 Consider the surface 
isogenous to a product $X = (B\times F)/\Delta_G$, where $\Delta_G$ is the diagonal of $G\times G$ acting naturally on $B\times F$.

Since the action is free, 
the invariants of $X$ are as follows:
\[
\kappa(X)=1,\,\chi_\topo(X)=\chi(\mo_X) = p_g(X)=0,\, q(X)=1,\, b_2(X) = \rho(X) =2
\]

The rational cohomology groups of $X$ are as follows:
\begin{equation}\label{eq: cohom}
 \small H^1(X,\QQ) = q^* H^1(F/G,\QQ), \, H^2(X,\QQ) = p^*H^2(B/G,\QQ)\oplus  q^*H^2(F/G, \QQ) 
\end{equation}
where $p\colon X\rightarrow B/G$ and $q\colon X\rightarrow F/G$ are the induced fibrations. 
 
Consider the following action of $\langle \tilde \sigma \rangle\cong \ZZ/n\ZZ$ on $B$:
\[
(x,y)\mapsto (\xi x, y),
\]
where $\xi$ is a primitive $n$-th root of unity.
Then $\tilde\sigma\times \id_F\in \Aut(B\times F)$ commutes with $\tau$, and hence it descends to an automorphism $\sigma\in \Aut(X)$, of order $n$. One sees immediately from Principle~\ref{SIPU}, or directly from  \eqref{eq: cohom} that $\sigma$ acts trivially on $H^*(X,\QQ)$, that is, $\sigma\in\Aut_\QQ(X)$. On the other hand, $\sigma$ permutes the $n$ double fibres of $p\colon X\rightarrow B/G$; see Figure~\ref{fig: perm}. Since  $n\geq 3$,  $\langle\sigma\rangle$ acts faithfully on $H^*(X,\ZZ)$ in view of Principle~\ref{prin: multiplefibres}.

\begin{figure}
\begin{tikzpicture}
\draw(-5.5,-2) -- (5.5,-2) node[right]{$\PP^1$};
\foreach \x in {-4, -2, 0, 2, 4}
\node at (\x, -2) [rbullet]{};
\node at (-4, -2) [label=below: $\xi$]{};
\node at (-2, -2) [label=below: $\xi^2$]{};
\node at (0, -2) [label=below: $\xi^3$]{};
\node at (2, -2.1) [label=below: $\cdots$]{};
\node at (4, -2) [label=below: $\xi^{n}$]{};
\foreach \x in {-4, -2, 0, 2, 4}
\draw[line width = 2pt, color=red] (\x, 0) -- (\x,3);
\node at (-4, 3) [red, label=above: $2F_1$]{};
\node at (-2, 3) [color=red, label=above: $2F_2$]{};
\node at (0, 3) [color=red, label=above: $2F_3$]{};
\node at (2, 3) [color=red, label=above: $\cdots$]{};
\node at (4, 3) [color=red, label=above: $2F_n$]{};
\foreach \x in {-5, -3, -1, 1, 3, 5}
\draw (\x, 0) -- (\x,3);
\node at (5.7, 0) {$X$};
\foreach \x in {-4, -2, 0, 2}
\draw[green, ->] (\x+.2, -.1) ..controls (\x+1, -.5) .. (\x+1.8, -.1)node[midway,below]{$\sigma$};
\draw[green, ->] (3.8, 3.5) ..controls (0, 4) .. (-3.8, 3.5)node[midway,above]{$\sigma$};
\end{tikzpicture}
\caption{ Cyclic permutation of double fibres}\label{fig: perm}
\end{figure}
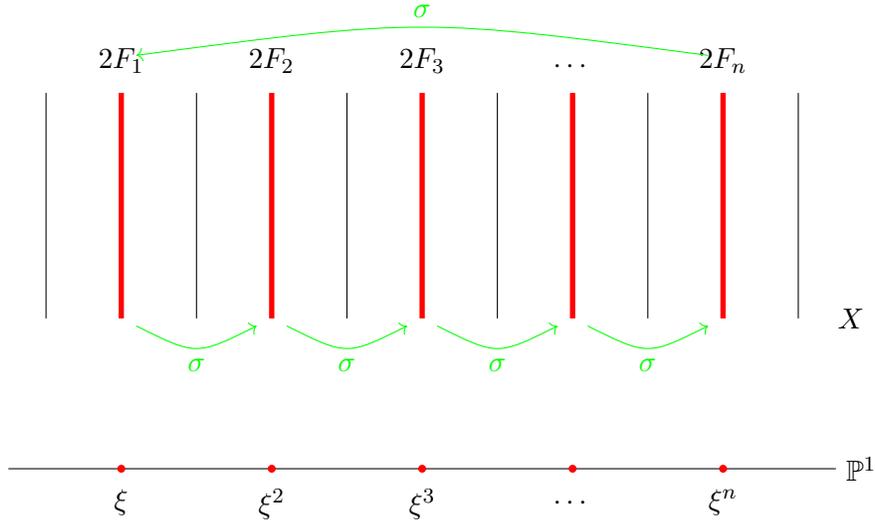

 It follows that
\[
[\Aut_\QQ(X):\Aut_\ZZ(X)] \geq 
 |\sigma|
 =n.
\]
Thus $[\Aut_\QQ(X):\Aut_\ZZ(X)]$ is not bounded, as $n$ goes to infinity.
\end{proof}

\section{Unbounded $[\Aut_\ZZ(X):\Aut_*(X)]$}
In this section, we construct a series of examples where $[\Aut_\ZZ(X):\Aut_*(X)]$
 is not bounded.

\begin{theo}\label{Kod1nonmin}
 For each positive integer $n$ there exists a  (non minimal) surface 
 $X$
 of Kodaira dimension $1$ such that 
 $\Aut_*(X)=\{\id_{X}\}$, and  
$$
\Aut_\ZZ(X)\cong \ZZ/n\ZZ.$$
\end{theo}
\begin{proof}

Let $C$ and $E$ be two smooth projective curves with $g(C)\geq 2$ and  $g(E)= 1$. Suppose that $G=\langle\sigma\rangle\cong\ZZ/n\ZZ$ acts faithfully on $C$ and $E$ in such a way  that
\begin{itemize}[leftmargin=*]
    \item $C^\sigma \neq \emptyset$ and $g(C/G)\geq 1$;
    \item $\sigma$ acts on $E$ by translations, that is, $\sigma(y)=y+a$ for some torsion element $a\in E$ of order  exactly $n$.
\end{itemize}
The diagonal $\Delta_G<G\times G$ acts freely on $C\times E$,  so we take the 
 SIP
 of unmixed type 
 $
 Y
 :=(C\times E)/\Delta_G$. 

 We have shown in (III) of Principle~\ref{SIPU} that $\Aut_\ZZ(Y) = \Aut_0(Y)$, and it consists of automorphisms that lift to an automorphism $\tilde \gamma$ of $C\times E$ of the form $\tilde\gamma(x,y)=(x, y+a)$ for some $a\in E$.
 
Now let $t\in C/G$ and  
 $X_t=\Bl_P(Y)$
 be the blow-up of a point $P\in F_t$,
  where $F_t$ denotes the fibre of the induced fibration $Y\rightarrow C/G$ over $t$. Then by 
Principle~\ref{prin: descend}
 we have
\[
\Aut_\ZZ(X_t) 
= \{\gamma\in
 \Aut_\ZZ(Y)
\mid \gamma(P)=P\}  \cong \Stab_G(t).
\]
Note that there exists a point $t_0$ with $
\Stab_G(t_0)
=G=\langle\sigma\rangle$ by the assumption that $C^\sigma\neq \emptyset$. 

 The situation is illustrated in Figure~\ref{fig: blow}, where $E$ denotes the exceptional divisor of the blow-up at a point of the fibre $F_{t_0}$:
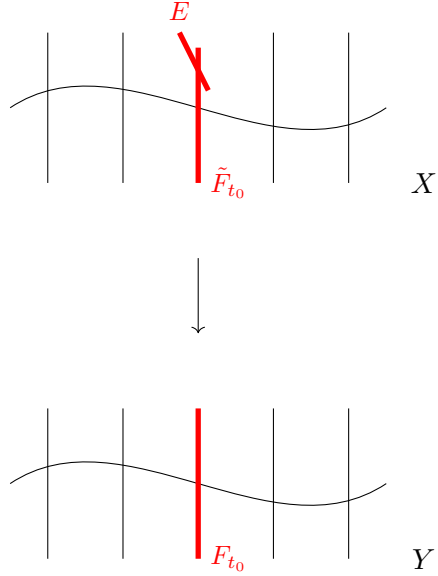
\begin{figure}
\begin{tikzpicture}
\begin{scope}[yshift=3cm]
\draw(-.5,0) .. controls (1,1) and (3,-1)  .. (4.5,0);
\foreach \x in {0,1,3,4}
\draw (\x, -1) -- (\x,1);
\draw[shorten >=-.3cm, line width = 2pt, color=red, ] (1.75, 1)node[above]{\small $E$ } -- (2,.5);
\draw[shorten >=-.3cm, line width = 2pt, color=red, ] (2,-1)node[right]{\small $\tilde F_{t_0}$ } -- (2,.5);
\node at (5, -1) {$X$};
\end{scope}
\draw[->] (2, 1) -- (2, 0);
\begin{scope}[yshift=-2cm]
\draw(-.5,0) .. controls (1,1) and (3,-1)  .. (4.5,0);
\foreach \x in {0,1,3,4}
\draw (\x, -1) -- (\x,1);
\draw[line width = 2pt, color=red] (2, -1) node[right]{\small $F_{t_0}$}-- (2,1);
\node at (5, -1) {$Y$};
\end{scope}
\end{tikzpicture}
\caption{ Blow-up of a point on a multiple fibre}\label{fig: blow}
\end{figure}
Letting $n=|G|$ go to infinity, we see that 
$|\Aut_\ZZ(X_{t_0})|$
is unbounded. 

 Now the proof is completed by applying Principle~\ref{prin: rigidity} that the group 
$\Aut_*(X_{t})$
is trivial for any $t$.
\end{proof}

\begin{cor}
Let  $X$
 be as in Theorem~\ref{Kod1nonmin}. Then $\Aut_*(X)=\{\id_X\}$.

As a consequence
\[
 [\Aut_\ZZ(X):\Aut_*(X)] = |\Aut_\ZZ(X)| 
\]
 is unbounded.
\end{cor}

\section{Unbounded $[\Aut_*(X):\Aut_0(X)]$}
In this section, we give examples of smooth projective surfaces admitting groups of $C^\infty$-isotopically trivial automorphisms
with  number of connected components which is not bounded. A classification of these  would be  desirable.

\begin{theo}\label{rational}
For each positive integer $m$ there exists a rational surface $X$, blow up of $\PP^2$, such that $\Aut_{\QQ}(X) = \Aut_*(X) \cong \ZZ/ m\ZZ$.
\end{theo}

A simple idea lies behind the construction: if we blow up a point $P$  in a complex manifold $Y$, the differentiable manifold we obtain
does not depend on the choice of the given point, as two simple arguments show: the first is that the diffeomorphism group 
 $\Diff(Y)$
acts transitively on the manifold $Y$, the second is  that,  varying the point $P$,
we get a family with base $Y$
of blow ups of $Y$, and they are all diffeomorphic by Ehresmann's theorem.

\begin{proof}

Let $P_1=(1:0:0),\, P_2=(0:1:0), \, P_3=(0:0:1)$ be the coordinate points of $\PP^2$, and let $P_4=(1:1:0)$. 
 Denote the lines connecting these points by $L_1=\overline{P_2P_3}$, $L_2=\overline{P_1P_3}$, $L_3=\overline{P_1P_2}$ and $L_4=\overline{P_3P_4}$.
 Let $\pi\colon X_4 \rightarrow \PP^2$ be the blow-up of the four points $P_i, 1\leq i\leq 4$; see Figure~\ref{fig: blow4}.
\begin{figure}
\begin{tikzpicture}[font=\tiny]
\begin{scope}[xshift=-3cm, scale=.8]
\draw (0,0) node (P41){};
\draw (0,1) node (P42){};
\draw (-1,1) node (P11){};
\draw (-1,2) node (P12){};
\draw (1,-1) node (P21){};
\draw (1,0) node (P22){};
\draw (2,.5) node (P31){};
\draw (2,1) node (P32){};
\draw (2,1.5) node (P33){};
\draw[shorten >=-.3cm, shorten <=-.3cm, red] (P11) --node[midway,right]{$E_1$} (P12);
\draw[shorten >=-.3cm, shorten <=-.3cm, red] (P21) --node[midway,right]{$E_2$} (P22);
\draw[shorten >=-.5cm, shorten <=-.5cm, red] (P31) -- (P33)node[above=.3cm]{$E_3$} ;
\draw[shorten >=-.3cm, shorten <=-.3cm, red] (P41) -- node[midway,right]{$E_4$} (P42);
\draw[shorten >=-.3cm, shorten <=-.3cm] (P11) -- node[midway,left]{$L_{3,4}$} (P21);
\draw[shorten >=-.3cm, shorten <=-.3cm] (P42) --node[midway, above=-.1cm]{$L_{4,4}$} (P32);
\draw[shorten >=-.3cm, shorten <=-.3cm] (P22) -- node[midway,above]{$L_{1,4}$} (P31);
\draw[shorten >=-.3cm, shorten <=-.3cm] (P12) --node[midway,above]{$L_{2, 4}$} (P33);
\node at (0.5, -1.5){\large  $X_4$};
\end{scope}
\draw[->] (-.5,.5) -- node[midway, below]{Blow up $\{P_i\}$}node[midway, above]{$\pi$}(1,.5);
\begin{scope}[xshift=4cm]
\draw (-1, 0) node(P1) {}node [left,align=center] {$ (1:0:0) P_1$};
\draw (1,  0) node (P2){}node [right, align=center] {$P_2 (0:1:0)$};
\draw (0, 1.5) node (P3){}node [right, align=center] {$P_3 (0:0:1)$};
\draw (0, 0) node[inner sep=0pt] (P4) {}node [below] {$P_4(1:1:0)$};
\draw[shorten <=-.3cm] (P1) --node[midway, above=-.1cm]{$L_{3}$} (P4);
\draw[ shorten >=-.3cm] (P4) -- (P2);
\draw[shorten >=-.3cm, shorten <=-.3cm] (P3) -- node[midway, right=-.1cm]{$L_{1}$}(P2);
\draw[shorten >=-.3cm, shorten <=-.3cm] (P3) -- node[midway, left=-.1cm]{$L_{2}$} (P1);
\draw[shorten >=-.1cm, shorten <=-.3cm] (P3)--node[midway, right=-.1cm]{$L_{4}$}(P4);
\foreach \x in {1,...,4}
\fill [red] (P\x) circle [radius=1pt];
\node at (0, -1){\large $\PP^2$} ;
\end{scope}
\end{tikzpicture}
\caption{Blow-up $X_4$  of  the plane: for each $1\leq i \leq 4$, $L_{i,4}$ on the left picture denotes the strict transform of $L_{i}$ on $X_4$ , and $E_i$ denotes the exceptional curve over $P_i$.}\label{fig: blow4}
\end{figure}
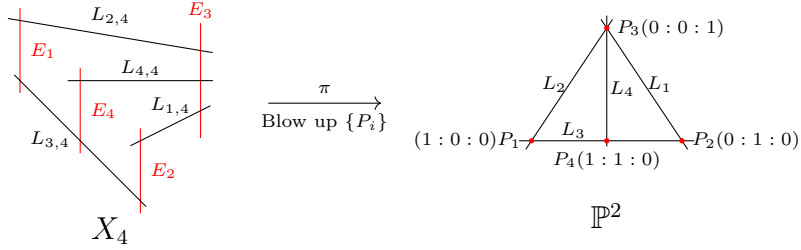

%\begin{center}
%\begin{tikzpicture}
%\draw (-1, 0) node (P1) {$\bullet$}node [above left] {$P_1$};
%\draw (0, 0) node (P4) {$\bullet$}node [below left] {$P_4$};
%\draw (1,  0) node (P2){$\bullet$}node [above right] {$P_2$};
%\draw (0, 1.5) node (P3){$\bullet$}node [left] {$P_3$};
%\draw (-1.5, 0) -- (1.5, 0);
%\draw (-1.2, -0.3) -- (0.2, 1.8);
%\draw (0,-.2)--(0,1.8);
%\draw (-.2,1.8)--(1.2,-.3);
%%\draw (0,2) -- (0,-.5);
%\end{tikzpicture}
%\end{center}

\begin{lem}
Let $G_4=\{\sigma\in \PP \GL(3) \mid \sigma(P_i)=P_i \text{ for } 1\leq i\leq 4\}.$ Then
\[
  \Aut_\QQ(X_4) \cong G_4
= \left\{
\begin{bmatrix}
1 & 0 & 0\\
0 & 1 & 0 \\
0 & 0 & a
\end{bmatrix}
 \biggm|a\in \CC^* = \CC-\{0\}
\right\}
\]
\end{lem}
\begin{proof}
The isomorphism is by 
 Principle~\ref{prin: descend}.
Since  $G_4$ fixes three distinct points $P_1, P_2, P_4$ on the line $L_3=(x_3=0)$, it acts as the identity on $L_3$. It follows that any element $\sigma\in G_4$ takes the form
\[
\begin{bmatrix}
1 & 0 & *\\
0 & 1 & * \\
0 & 0 & * 
\end{bmatrix}
\]
where $*$ denotes entries to be determined. Since $\sigma$ fixes $P_3=(0:1:0)$, one sees then easily that 
\[
\sigma = \begin{bmatrix}
1 & 0 & 0\\
0 & 1 & 0 \\
0 & 0 & a
\end{bmatrix}
\]
for some $a\in 
 \CC^*
$.
\end{proof}
Blow up now  $P_5$, infinitely near $P_4$, in the direction of the line 
$
 L_{4}
= \overline{P_3P_4}=(x_0-x_1=0)$. Since $G_4$ fixes the point  $P_5$, 
 the action of
$G_4$ lifts to  $X_5=\Bl_{P_5}(X_4)$.  Continue to blow up inductively  $P_{n+1}\in E_n\cap 
 L_{4,n}
\subset X_{n}$ for $n\geq 4$,  where  
 $L_{4,n}$   is the strict transform of  
 $L_{4}$
on  $X_{n}$ so that we get a chain of surfaces
\[
 X_{n}\rightarrow\cdots \rightarrow X_6 \rightarrow X_{5} \rightarrow X_4 
\]
such that  $G_4=G_5=G_6=\cdots=G_n$ acts on them
 equivariantly. The situation is illustrated by Figure~\ref{fig: blown}, 
 %\ref{iterated-blow}: 
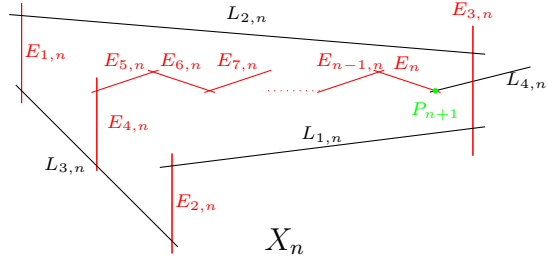
\begin{figure}\label{iterated-blow}
\begin{tikzpicture}[font=\tiny]
\begin{scope}
\draw (0,0) node (P41){};
\draw (0,1) node[inner sep=0] (P42){};
\draw (-1,1) node (P11){};
\draw (-1,2) node (P12){};
\draw (1,-1) node (P21){};
\draw (1,0) node (P22){};
\draw (5,.5) node (P31){};
\draw (5,1) node (P32){};
\draw (5,1.5) node (P33){};
\draw (P42)+(.75,.25) node[inner sep=0] (P5){};
\draw (P5)+(.75,-.25) node[inner sep=0](P6){};
\draw (P6)+(.75,.25) node[inner sep=0](P7){};
\draw (P7)+(0,-.25) node[inner sep=0](P8){};
\draw (P8)+(.75, 0) node[inner sep=0](P9){};
\draw (P9)+(.75, .25) node[inner sep=0](P10){};
\draw (P10)+(.75, -.25) node[inner sep=0, label=below:{\color{green} $P_{n+1}$}](P101){};
\draw (P101)+(1, .25) node[inner sep=0](P102){};
\draw[red, shorten >=-.1cm, shorten <=-.1cm] (P42)--node[midway, above]{$E_{5,n}$}(P5);
\draw[red, shorten >=-.1cm, shorten <=-.1cm] (P5)--node[midway, above]{$E_{6,n}$}(P6);
\draw[red, shorten >=-.1cm, shorten <=-.1cm] (P6)--node[midway, above]{$E_{7,n}$}(P7);
\draw[red, dotted] (P8)--(P9);
\draw[red, shorten >=-.1cm, shorten <=-.1cm](P9)--node[midway, above]{$E_{n-1, n}$}(P10);
\draw[red, shorten >=-.1cm, shorten <=-.1cm](P10)--node[midway, above]{$E_n$}(P101);
\draw[shorten >=-.3cm, shorten <=-.1cm](P101)--node[right=.3cm]{$L_{4,n}$} (P102);
\draw[shorten >=-.3cm, shorten <=-.3cm, red] (P11) --node[midway, right=-.1cm]{$E_{1,n}$} (P12);
\draw[shorten >=-.3cm, shorten <=-.3cm, red] (P21) --node[midway, right=-.1cm]{$E_{2,n}$} (P22);
\draw[shorten >=-.5cm, shorten <=-.5cm, red] (P31) -- (P33)node[above=.3cm]{$E_{3,n}$};
\draw[shorten >=-.2cm, shorten <=-.2cm, red] (P41) -- node[midway, right]{$E_{4,n}$}(P42);
\draw[shorten >=-.3cm, shorten <=-.3cm, red] (P11) -- (P12);
\draw[shorten >=-.3cm, shorten <=-.3cm, red] (P21) -- (P22);
\draw[shorten >=-.5cm, shorten <=-.5cm, red] (P31) -- (P33);
\draw[shorten >=-.2cm, shorten <=-.2cm, red] (P41) -- (P42);
\draw[shorten >=-.3cm, shorten <=-.3cm] (P11) -- node[midway, left]{$L_{3,n}$} (P21);
\draw[shorten >=-.3cm, shorten <=-.3cm] (P22) -- node[midway,above=-.1cm]{$L_{1,n}$} (P31);
\draw[shorten >=-.3cm, shorten <=-.3cm] (P12) --node[midway,above]{$L_{2,n}$} (P33);
\node at (2.5, -1){\large $X_{n}$} ;
%\draw[->, green] (P101) --node[below=.1cm, left=-.2cm]{$x_n$} ($(P101)!.5cm!(P10)$);
%\draw[->, green] (P101) --node[right, below=-.1cm]{$y$} ($(P101)!.5cm!(P102)$);
\fill[green] (P101) circle[radius=1pt];
\end{scope}
\end{tikzpicture}
\caption{Iterated blow-up of the plane}\label{fig: blown}
\end{figure}
 where $L_{i,n}$ denote the strict transform of $L_{i}$ on $X_{n}$ for $1\leq i \leq 4$, and $E_{k,n}$, $1\leq k\leq n-1$, is the strict transform of the exceptional curve $E_k$ of $X_{k}\rightarrow X_{k-1}$ on $X_{n}$.

Now, let us look at the blown-up point $ P_{n+1}\in X_{n}$, $n\geq 4$. An element of  $G_n=G_4$ can be written as 
\[
\sigma_a  = 
\begin{bmatrix}
1 & 0 & 0\\
0 & 1 & 0 \\
0 & 0 & a
\end{bmatrix},\, a\in \CC^*
\]
It fixes the curve $L_{4,n}$ and preserves the exceptional curve $E_n$  of the blow-up $X_n\rightarrow X_{n-1}$
 for $n\geq 4$. At $P_4\in \PP^2$ there are local coordinates $(x,y)$ such that 
\[
\sigma_a(x,y)=(x,ay). 
\]
At  $P_5\in X_4$ there are local coordinates $(x/y, y)$, with $L_{4,4}=(x/y=0)$ and 
$ E_4=(y=0)$, such that 
\[
\sigma_a(x/y, y) = ((1/a)(x/y),ay)
\]
In general, at the point  $P_{n+5}\in X_{n+4}$ with $n\geq 1$ there are local coordinates $(x/y^{n+1}, y)$ such that $L_{4, n+4} =(x/y^{n+1}=0)$ and $E_{n+4} =(y=0)$, and
\[
\sigma_a(x/y, y) = ((1/a^{n+1})(x/y^{n+1}),ay)
\]
Observe that the local coordinate of $P_{n+5}\in E_{n+4}$ is $x/y^{n+1}$, so $\sigma_a$ acts as the  identity on  $E_{n+4}$ if and only if $a^{n+1}=1$; the two intersection points 
 $P_{n+5}, P_{n+5}'$ of $E_{n+4}$ with $L_{4, n+4}$ and the strict transform of $E_{n+3}$  are fixed by $\sigma_a$ for any $a\in \CC^*$. 

For any  $P\in E_{n+4}$, we have by 
 Principle~\ref{prin: descend}
\[
\Aut_{\QQ}( \Bl_P(X_{n+5}))= \{\sigma\in  G_{n+4}\mid \sigma(P) =P\}=
\begin{cases}
G_4
 & \text{ if } 
 P=P_{n+5} \text{ or } P_{n+5}' \\
\mu_{n+1} & \text{ otherwise}
\end{cases}
\]
where $\mu_{n+1}$ is the cyclic group of order $n+1$. As $P$ varies, we obtain a family of surfaces $\Phi\colon \mx\rightarrow 
 E_{n+4}\cong\PP^1$ 
 such that the fibre 
 $\Phi^*P = \Bl_P(X_{n+5})$. 
 
 These are all diffeomorphic, and there is a diffeomorphism 
by which $\Aut_{\QQ}(\Bl_P(X_{n+5})) $
 is always a subgroup of the
same group  $G_4$. 
Hence
\[
 \Aut_{\QQ}(\Bl_P(X_{n+5})) = \Aut_{*}(\Bl_P(X_{n+5})).
 \]
For 
 $P\in E_{n+4}-\{P_{n+5}, P_{n+5}'\}$
 we have 
  $\Aut_0(\Bl_P(X_{n+5}))=\{\id\}$,
  and
\[
[\Aut_{*}(\Bl_P(X_{n+5})): \Aut_0(\Bl_P(X_{n+5}))] = |\Aut_{*}(\Bl_P(X_{n+5}))| = n+1
\]
which can be arbitrarily large.
\end{proof}

\begin{rmk}
(1) The above construction, exploring the difference between $\Aut_0(\Bl_P(X))$ and $\Aut_0(X)$, shows that the statement in the fourth paragraph of \cite[page 251]{Pe80} is wrong.

\end{rmk}
Completely similar examples can be constructed for other non minimal ruled surfaces, blowing up appropriately 
 any decomposable
 $\PP^1$-bundle over  any curve $C$ of arbitrary genus. 

Indeed, generalizing Theorem~\ref{rational}, we obtain a recipe for constructing surfaces with $[\Aut_*(X):\Aut_0(X)]$ arbitrarily large: 
\begin{enumerate}[leftmargin=*]
\item Choose a $\CC^*$-surface  $Z$
 that is, a projective smooth surface with $\CC^*\cong G<  \Aut_0(Z)$. For example, we can take  $Z$
 to be any smooth toric surface or any decomposable $\PP^1$-bundle over a curve.
\item Let $Y\rightarrow Z$
be a composition of blow-ups of $G$-fixed points such that $\Aut_0(Y)=G$.
These blow-ups kill the automorphisms in  $\Aut_0(Z)\setminus G$
 while preserving the action of $G$.
\item Blow up a point $P\in Y$
and its infinitely near points that is fixed by the whole $G$,  in the same way as in the proof of Theorem~\ref{rational}:
\[
X_{n}\rightarrow\cdots \rightarrow X_{2} \rightarrow X_{1} \rightarrow Y 
\]
The exceptional curve $E_n$ of the $n$-th blow-up $X_{n}\rightarrow X_{n-1}$
 is then invariant under $G$ and is fixed by and only by a finite cyclic subgroup $\ZZ/n\ZZ\cong H<G$. 
\item Let $X\rightarrow X_{n}
$ be the blow-up of a general point $ P\in E_n$. Then it holds
\[
\Aut_0(X) =\{\id_X\}, \text{ and } \Aut_*(X) \cong \ZZ/n\ZZ.
\]
\end{enumerate}
The method can also be used to construct $X$ such that $\dim \Aut_0(X)>0$ and $[\Aut_*(X):\Aut_0(X)]$ is arbitrarily large.

\bigskip

 Surprisingly, however, the same unboundedness phenomenon for $\Ga_*(X): = \Aut_*(X) / \Aut_0(X)$ can happen also for minimal ruled surfaces:
 
 \begin{theo}\label{thm: ell rule unbounded}
 Let $E$ be an elliptic curve and let $X : = \PP (\hol_E \oplus \hol_E(D))$ where $D$ is a divisor of even positive degree $d  = 2m > 0$.
 
 Then $\Ga_*(X) $ surjects onto a subgroup of index at most $4$ inside $(\ZZ/d\ZZ)^2 $, in particular 
 $$|\Ga_*(X) | \geq m^2.$$
 \end{theo}
 \begin{proof}

Consider the vector bundle  $V := \hol_E \oplus \hol_E(D)$, so that   $X := \PP(V)$.

Any $\ga \in \Aut(X)$ preserves the fibration $ f \colon X \ra E$, in particular $\ga$ acts on $E$. If $\ga \in \Aut_{\QQ}(X)$,
then necessarily $\ga$ acts on $E$ by a translation $\tau_a, \ a \in E$. Since $ \tau_a^* (X) \cong X$,
it must be that  $ \tau_a^* (V) \cong V \otimes L$, for a suitable line bundle $L$, say, of degree $l$,
 on $E$.
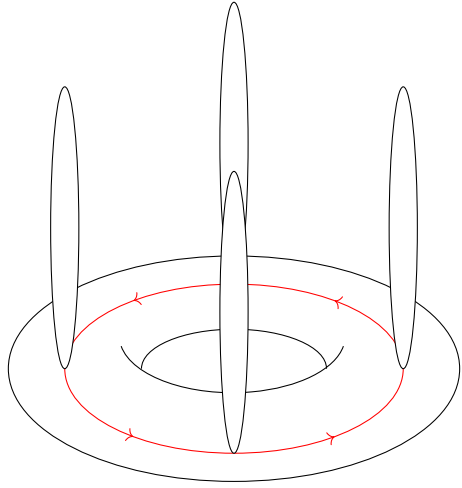
\begin{figure*}
\begin{tikzpicture}[scale=.75]
\begin{scope}[decoration={markings, mark=at position 0.5 with {\arrow{>}}}] 
\draw[] (0,0) ellipse [x radius=4cm, y radius=2cm];
\draw[] (-2,.4) arc [start angle=190, end angle=350, x radius=2cm, y radius=1cm];
\draw[] (1.64,0) arc [start angle=0, end angle=180, x radius=1.64cm, y radius=.7cm];
\draw[red, postaction={decorate}] (-3,0) arc [start angle=180, end angle=270, x radius=3cm, y radius=1.5cm];
\draw[red, postaction={decorate}] (0,-1.5) arc [start angle=270, end angle=360, x radius=3cm, y radius=1.5cm];
\draw[red, postaction={decorate}] (3,0) arc [start angle=0, end angle=90, x radius=3cm, y radius=1.5cm];
\draw[red, postaction={decorate}] (0,1.5) arc [start angle=90, end angle=180, x radius=3cm, y radius=1.5cm];
\end{scope}
\begin{scope} 
\fill[white] (-3, 2.5) ellipse [x radius=.25cm, y radius=2.5cm];
\draw (-3, 2.5) ellipse [x radius=.25cm, y radius=2.5cm];
\end{scope}
\begin{scope}[xshift=3cm, yshift=1.5cm]
\fill[white] (-3, 2.5) ellipse [x radius=.25cm, y radius=2.5cm];
\draw (-3, 2.5) ellipse [x radius=.25cm, y radius=2.5cm];
\end{scope}
\begin{scope}[xshift=6cm]
\fill[white] (-3, 2.5) ellipse [x radius=.25cm, y radius=2.5cm];
\draw (-3, 2.5) ellipse [x radius=.25cm, y radius=2.5cm];
\end{scope}
\begin{scope}[xshift=3cm, yshift=-1.5cm]
\fill[white] (-3, 2.5) ellipse [x radius=.25cm, y radius=2.5cm];
\draw (-3, 2.5) ellipse [x radius=.25cm, y radius=2.5cm];
\end{scope}
\end{tikzpicture}
\caption{ An automorphism of an elliptic ruled surface inducing translation on the base.}
\end{figure*}

Because from the degrees equality (they are in increasing order) $(0,d) = (l, d + l ) \Rightarrow l = 0$, we get that $L$ is trivial, and $ \tau_a^* ( \hol_E(D)) \cong  \hol_E(D)$,
hence $$ a \in \sK : = \Ker(\Phi_D : E \ra E) \cong (\ZZ/d\ZZ)^2,$$ where  $\Phi_D (a) = \tau_a^* (D) - D \in E$.

We have the exact sequence 
$$ 1 \ra \PP \GL(V) \ra  \Aut_{\QQ}(X) \ra \sK \ra 1$$
where the group $\PP \GL(V)$ is connected, since it consists of the linear maps 
$$ (v_1, v_2) \mapsto (v_1, \be v_1 + \la v_2), \beta \in H^0 (E, \hol_E(D)), \la \in \CC^*.$$ 

There remains to show that  a subgroup of index at most $4$ of $\Aut_{\QQ}(X)$ is contained in $ \Aut_*(X)$, so that 
this subgroup maps via $\Aut_{\QQ}(X) \ra \sK$ to a subgroup
of index at most $4$ in $\sK$.

To prove this, observe that, differentiably, $V$ is classified by $c_1(V)$ (it is the pull back of a classifying map to a
Grassmannian of $2$-planes).  Hence $V$ is topologically equivalent to $ \hol_E(D')\oplus  \hol_E(D')$, where $D'$ is a divisor of degree $=m$.

Hence $X$ is differentiably equivalent to the product $   \PP^1 \times E$. 

And $\Aut_{\QQ}(X)$ maps to the group of 
 vector bundle diffeomorphisms, so that $\ga (v,t) = (A(t) v, t + a)$, with $ a \in \sK$, and $A : E \ra \PP \GL(2, \CC)$.
 
 Hence $\ga$ is isotopic to $\ga_0 (v,t) : = (A(t) v, t )$. Now,
 $$A : E \ra \PP GL(2, \CC) = \PP SU(2) = SO(3) / \pm 1$$
  lifts to $SO(3) \cong S^3$ if and only if $\pi_1(A) : \pi_1(E) \ra \pi_1 (\PP SU(2) ) = \ZZ/2\ZZ$ is trivial. 
  
  The group $\sG : = \Hom ( \pi_1(E), \ZZ/2\ZZ ) \cong (\ZZ/2\ZZ)^2$ has exactly $4$ elements, and we have therefore
  a homomorphism $\pi: \Aut_{\QQ}(X) \ra \sG$. For elements   $\ga \in \Ker (\pi)$, we get diffeomorphisms 
  such  that the corresponding map $A$ has now 
  $\pi_1(A)=0$.
  
  Therefore $\ga \in \Ker (\pi)$ yields a map $A$ which lifts to a differentiable map $A'$ to the three-sphere $S^3$. This,   by Sard's lemma, omits one point $P_0$.
  Since  $S^3 \setminus \{ P_0\} \cong \RR^3$, which is contractible, we obtain that $A$ and $\ga$ are isotopic to the identity.
\end{proof} 

The case of more general $\PP^1$-bundles over curves shall be treated in the next sections.

\section{Cohomologically trivial automorphisms of rational surfaces}

In a sense, we are now  going first to explain the philosophy behind the construction of Theorem \ref{rational}.

Moreover, we sketch how a similar procedure leads to examples where $\Aut_0(X)$ is nontrivial and the
group of components $\Ga_*(X)$ can be arbitrarily large.

Let $X$ be a blow up of the projective plane $\PP^2$, obtained by first blowing up $P_1, \dots, P_r \in \PP^2$,
and then infinitely near points $P_{r+1}, \dots, P_k$. We assume that $k \geq 2$, else we have the plane or the 
$\PP^1$-bundle $ \FF_1 : = \PP (\hol_{\PP^1} \oplus \hol_{\PP^1}(1))$, and for these $\Aut(X) = \Aut_0 (X)$.

Then we define inductively: 
$$ G_0 : = \PP \GL (3, \CC), \ G_1 : = \{ g \in G_0 \mid g (P_1 ) = P_1\}, \ G_{i+1} : = \{ g \in G_i \mid g (P_{i+1} ) = P_{i+1}\},$$
and $ X_0 : = \PP^2, X_{1}, \dots , X_{k} = X$.

Here $G_i$ acts on $X_i$ and $ G_{i+1}$ is the fibre over $ P_{i+1}$ of $G_i \ra G_i  P_{i+1}$; notice that stabilizers
of points lying in a fixed orbit are conjugate, hence all fibres are smooth and  isomorphic, in particular they
are connected if the orbit is 1-connected, a fortiori if $G_i$ is 1-connected. 

Clearly  $ G_{i+1} = G_i$ if the orbit has dimension $0$, and 
$G_{i+1}$ is connected if the orbit is isomorphic to $\CC$ or $\PP^1$. 
But 
$G_{i+1}$ may not be connected if the orbit is isomorphic to $\CC^*$.

Notice that $G_1$ is isomorphic to  $ \Aff (2, \CC)$, while if $r \geq 2$, then $G_2$ is isomorphic to the subgroup
of the affine group
 $ \Aff (2, \CC)$ of transformations $ v \ra A v + w$ such that $e_1$ and $e_2$  are eigenvectors of the matrix $A$
 (put the two points at infinity).
 
 If $ r \geq 3$, then $G_3$ is isomorphic to $\CC^* \times \CC^*$ if the three points are not collinear, else we get
 the group  $\CC^2 \rtimes \CC^*$ of dilations $ v \ra \la v + w$ ($\la \in \CC^*, w \in \CC^2$).
 
We can briefly summarize the situation as follows. If $X$ is a blow up of $\PP^2$, then $G_k = \Aut_{\QQ}(X) $ is:
 \begin{enumerate}
 \item
 trivial if $r \geq 4 $ and $\{P_1, \dots, P_r\}$  contains a projective basis,
 \item
 a subgroup of $\CC^* \times \CC^*$ if $ r \geq 3$ and $\{P_1, \dots, P_r\}$ contains three
 non collinear points, here $G_k$ does not need to be connected if $k \geq 5$;
 \item
 a subgroup of $\CC^2 \rtimes \CC^*$ if all the points $P_1, \dots, P_r $ are collinear and $r \geq 3$,  
 again here $G_k$ does not need to be connected if $k \geq 5$;
 \item
 $G_k$ is connected if $k \in \{1,2\}$, and for $k=3$ (both $r=2$ and $r=1$);
 \item
 $G_4$ may be   disconnected for $k=4$, $r=1$. 
 
 \end{enumerate}
 
To conclude with these general observations, we observe that
if  $X$ is a blow up of $\PP^2$, then $\Aut_{\QQ}(X) = \Aut_{*}(X)$.
Indeed the same statement holds for all rational surfaces, and a proof 
could be done  by induction on $k$ for $\Aut_{\QQ}(X)= G_k$
provided one could show the following general assertion: 

\begin{question}
Let $X$ be a compact complex manifold and $\s \in \Aut(X)$ an automorphism which admits a fixed point $P$
and which is differentiably isotopic to the identity. Let $Z $ be the blow up of $X$ at $P$, and $\s'$
the induced automorphism of $Z$. Is $\s'$ differentiably isotopic to the identity?
\end{question}
Following the ideas we have introduced, we shall give a proof of the next theorem, which  can be viewed as a cohomological characterization of $C^\infty$-isotopically trivial automorphisms for smooth projective rational surfaces.
\begin{thm}\label{thm: rational}
Let $X$ be a smooth projective rational surface. Then 
\[
\Aut_*(X)=\Aut_\ZZ(X) = \Aut_\QQ(X).
\]
\end{thm}

We need the following lemma for the proof of Theorem~\ref{thm: rational}.
\begin{lem}\label{lem: Gm}
Let $G$ be a connected linear algebraic group, defined over $\CC$. Let $H$ be an algebraic subgroup of $G$ and $H_0$ its identity component. Then, for any $\sigma\in H\setminus H_0$, there is an element $\sigma'\in \sigma H_0$ such that $\sigma'$ has finite order, and $\sigma'$ is contained a $1$-dimensional multiplicative subgroup $T\cong\CC^*$ of $G$.
\end{lem}
\begin{proof}
Let $\sigma=\sigma_s\sigma_u$ be the Jordan decomposition in $H$, with $\sigma_s$ semisimple and $\sigma_u$ unipotent (\cite[Theorem~2.4.8]{Spr98}, \cite{humphreys} 15.3, page 99). Necessarily $\sigma_u\in H_0$, because the Zariski closure of the subgroup generated by $\sigma_u$ is an additive group $\GG_a \cong \CC$ (see  \cite{humphreys},  15.5 exercise 9, page 101).

Hence $\sigma_s\in \sigma G_0$. We can replace $\sigma$ by $\sigma_s$, and assume that $\sigma$ is semi-simple. Since $H$ is algebraic, there exists an $n\in\ZZ_{>0}$ such that $\gamma:=\sigma^n \in H_0$. Then (see \cite{humphreys}, 22.2, Cor. A, page 140, and Prop. 19.2, pages 122-123)
there is a maximal torus $T_H$ of $H_0$ containing $\gamma$. Note that $T_H$ is divisible and commutative. Therefore, there exists an $\tau\in T_H$ such that $\sigma^n= \tau^n$ and $\sigma\tau=\tau\sigma$ follows since $\s$ commutes with the whole 1-parameter subgroup generated by $\ga$. Then $\sigma':=\sigma\tau^{-1}\in \sigma T_H\subset \sigma H_0$ is of finite order. 

Since $G$ is connected, there is a maximal torus $T_G$ of $G$ containing $\sigma'$. Since $\sigma'$ has finite order, one sees easily that there is a $1$-dimensional multiplicative group $T\cong\CC^*$ of $T_G$ containing $\sigma'$.
\end{proof}

\begin{proof}[Proof of Theorem~\ref{thm: rational}]
It suffices to show that $\Aut_*(X) = \Aut_\QQ(X)$. 

Let $\sigma\in \Aut_\QQ(X)$ be a numerically trivial automorphism. Without loss of generality, we can assume that $\sigma\notin \Aut_0(X)$. 

 Let $f\colon X\rightarrow Y$  be a birational morphism to a smooth minimal model. Since $f$ is a composition of blow-downs of $(-1)$-curves, which are preserved by numerically trivial automorphisms (Principle~\ref{prin: negative}), the automorphism $\sigma$ descends all along to an automorphism $\sigma_Y\in \Aut_\QQ(Y)$. In fact, we have $\Aut_\QQ(X)\subset \Aut_\QQ(Y)$, viewed as subgroups of the birational automorphism group $\Bir(X)=\Bir(Y)$. Note that $Y$ is either $\PP^2$ or one of the Segre-Hirzebruch surfaces, and it holds $\Aut_\QQ(Y)=\Aut_0(Y)$.

 By Lemma~\ref{lem: Gm}, after replacing $\sigma$ by an automorphism $\sigma'$ in the same component  of $\Aut_\QQ(X)$, there is a subgroup $T\cong\CC^*$ of $\Aut_\QQ(Y)$ containing $\sigma$.  Factor $f\colon X\rightarrow Y$ into
\[
f\colon X= X_{n}\xrightarrow{f_{n}} X_{n-1} \xrightarrow{f_{n-1}}\cdots \rightarrow X_{1} \xrightarrow{f_{1}} X_{0} =Y
\]
so that $f_{i+1}\colon X_{i+1}\rightarrow X_{i}$ blows up an $\Aut_\QQ(X)$-fixed point $P_{i}\in X_{i}$. There is some $0\leq i\leq n-1$ such that
\begin{itemize}[leftmargin=*]
\item $T<\Aut(X_{i})$, but
\item $P_{i}\in X_{i}$ is not fixed by the whole $T$,
\end{itemize}
so the action of $T$ does not lift to $X_{i+1}$. The orbit of $P_{i}$ under the action of $T$ is  
 $TP_{i}\cong \CC^*.$
The closure $\overline{TP_{i}}$ is an irreducible rational curve with $T$ acting on it, and the normalization $\nu\colon\PP^1\rightarrow \overline{TP_{i}}$ is $T$-equivariant. Let $\Gamma_\nu\subset \PP^1\times X_{i}$ be the graph of $\nu\colon \PP^1\rightarrow \overline{TP_{i}}\hookrightarrow X_{i}$. Note that $T$ acts diagonally on $\PP^1\times X_{i}$ and $\Gamma_\nu$ is $T$-invariant. Blowing up $\PP^1\times X_{i}$ along $\Gamma_\nu$, we obtain a $T$-equivariant family 
\[
\Phi_{i+1}\colon \sX_{i+1} \xrightarrow{F_{i+1}} \PP^1\times X \rightarrow \PP^1
\]
The fibre of $\Phi_{i+1}$ over $P_{i}\in TP_{i}\subset \PP^1$ is just $X_{i+1}$, and  the action of $T$ lifts to the fibres over $0$ and $\infty$.

Now let us look at the blow-up $f_{i+2}\colon X_{i+2}\rightarrow X_{i+1}$. Recall that the blown-up point is $P_{i}\in X_{i+1} \subset \sX_{i+1}$. The orbit $TP_{i} \subset \sX_{i+1}$ is a section of $\Phi_{i+1}$ over $\PP^1\setminus \{0,\infty\}$, and it readily extends to a section $\sS_{i-1}$ over the whole $\PP^1$. Let $F_{i+2}\colon \sX_{i+2}\rightarrow \sX_{i+1}$ be the blow-up along $\sS_{i-1}\cong \PP^1$, which is $T$-equivariant, extending $f_{i+2}\colon X_{i+2}\rightarrow X_{i+1}$ to a family of blow-ups.

We can continue this way, extending each blow-up $f_{j+1}\colon X_{j+1}\rightarrow X_j$, $j\geq i$, to a $T$-equivariant family of blow-ups $F_{j+1}\colon \sX_{j+1}\rightarrow \sX_j$. In the end, we get a commutative diagram of $T$-equivariant morphisms:
\[
\begin{tikzcd}
\sX_{n} \arrow[r, "F_{n}"] \arrow[rrd,"\Phi_{n}"']&\sX_{n-1} \arrow[r, "F_{n-1}"] \arrow[rd] &\cdots \arrow[r] \arrow[d]& \sX_{i+1} \arrow[r, "F_{i+1}"]  \arrow[dl]& \sX_{i}  \arrow[dll, "\Phi_{i}"]\\
&& \PP^1 &&
\end{tikzcd}
\]
The $T$-equivariant family $\Phi_n\colon \sX_{n}\rightarrow \PP^1$ satisfies the following properties:
\begin{itemize}[leftmargin=*]
\item $X$ is the fibre of  $\Phi_{n}$ over the point $P_{i}\in \PP^1$;
\item The subgroup $H=\langle \Aut_0(X), \sigma \rangle$ of $\Aut_\QQ(X)$ acts fibrewise on $\sX_{n}$ and extends to an action of $T$ on the two fibres over $0$ and $\infty$.
\end{itemize}
Since $T\cong\CC^*$ is connected, we infer that $\sigma$ is $C^\infty$-isotopic to $\id_X$.
\end{proof}

\section{Cohomologically
trivial automorphisms of  surfaces according to Kodaira dimension}
In this section, we  begin to investigate more systematically the boundedness of $[\Aut_\QQ(X):\Aut_0(X)]$, looking through the Enriques--Kodaira classification of surfaces.

\subsection{Case $\kappa(X)=-\infty$}
In this subsection we study the case $\kappa(X)=-\infty$ and $X \neq \PP^2, \PP^1 \times \PP^1$. 

In fact, for $\PP^2$, $\Aut(X) = \Aut_0(X)$,
while, for $X= \PP^1 \times \PP^1$, $\Aut_0(X) <  \Aut(X)$ has index $2$ and  $\Aut_0(X) = \Aut_\QQ(X)$.

Let $B$ be a smooth projective curve and $f\colon X \rightarrow B$ be a $\PP^1$-bundle. Then  $X\cong \PP(\E)$ is the projectivization of some rank two vector bundle $\E\rightarrow B$. We denote by  $\me$  the sheaf of holomorphic sections of $\E$ and often do not distinguish between $\me$ and $\E$. We shall say that $X$ is decomposable as a ruled surface over $B$ if $\E$ is so. We have $\pi_1(X)\cong\pi_1(B)$ and $H_1(X,\ZZ) \cong H_1(B,\ZZ)$ which is torsion free, while $H^2(X,\ZZ) = \ZZ F \oplus \ZZ \Sigma$,
where $F$ is a fibre and $\Sigma$ is a section, hence $\Aut_\QQ(X) = \Aut_\ZZ(X)$.

Moreover, topologically and differentiably the bundle $\E$ is determined by its first Chern class (which determines the class of the map to the classifying space, the infinite Grassmannian $Gr_{\CC}(2, \infty)$).

Hence from the differentiable view point there are only two cases:
\begin{enumerate}
\item
$\deg (\sE)$ even: $X$ is diffeomorphic to $ \PP^1 \times B$
\item
$\deg (\sE)$ odd: $X$ is diffeomorphic to $\PP (\hol_B \oplus \hol_B(P))$, where $P$ is a point of $B$,
and it is obtained from $ \PP^1 \times B$ via an elementary transformation blowing up a point and then blowing down
the strict transform of the fibre through the point.
\item
$\deg (\sE)$ odd and $B$ has genus $\geq 1$: then there is an \'etale double covering $B' \ra B$ such that the
pull back $X'$ is diffeomorphic to $\PP^1 \times B'$.
\end{enumerate}

The following facts about $\Aut(X)$ can be found for instance  in \cite{Ma71}:
\begin{enumerate}[leftmargin=*]
\item If $X\neq \PP^1\times\PP^1$ then there is exactly one $\PP^1$-bundle structure on $X$, so we have an exact sequence
\[
1\rightarrow \Aut_B(X) \rightarrow \Aut(X)\ra \Aut(B).
\]
where $$\Aut_B(X) = \{\sigma\in\Aut(X)\mid \sigma \text{ preserves every fibre of } f\colon X\rightarrow B\}.$$
The image of $\Aut(X)\ra \Aut(B)$ is the group of automorphisms $\ga$ of $B$ such that $\ga^* (\sE) \cong \sE \otimes L$
for a suitable line bundle $ L \in \Pic^0(B)$.

\item The exact sequence of sheaves of Lie groups on $B$
$$ 1 \ra \hol_B^* \ra \GL(2, \hol_B) \ra \PP \GL(2, \hol_B) \ra 1$$
induces a short exact sequence
\[
1\rightarrow \Aut_B(\E)/\CC^* \rightarrow \Aut_B(X) \rightarrow \Delta\rightarrow 1
\]
where $\Aut_B(\E)$ denotes the automorphism group of the vector bundle $\E$ over $B$, $\CC^*$ acts on the fibres of $\E\rightarrow B$ by scalar multiplication, and $\Delta:=\{\ml\in\Pic^0(B)\mid \me\otimes\ml\cong\me\}$. 

Note  that the fact that the cokernel is $\Delta$, contained in $\Pic^0(B)[2]$, the $2$-torsion part of $\Pic^0(B)$, follows
 since every automorphism of $X$ preserves the
class of the  relative canonical divisor.

\item Let $e:= \max \{2\deg\ml -\deg \me \mid \ml\subset \me \text{ invertible subsheaf}\}$. Observe
that $- e$  is the minimal self intersection
of a section of $ X \ra B$, and that $\sE$ is stable if $ e < 0$.

Then there are the following possibilities for $\Aut_B(\E)$:
\begin{enumerate}[leftmargin=*]
\item If $e<0$, that is, $\me$ is stable, then $\Aut_B(\E) = \CC^*$.
\item If $\me$ is indecomposable, $\ml\subset\me$ is the (unique) invertible subsheaf of maximal degree, $r=h^0(B, \ml^{\otimes2}\otimes(\det\me)^{-1})$, then 
$\Aut_B(\E)  \cong H_r$, where 

\[
H_r :=\left\{\left(
\begin{pmatrix} \alpha & 0\\ 0 & \alpha \end{pmatrix},\begin{pmatrix} \alpha & t_1\\ 0 & \alpha \end{pmatrix},\cdots, \begin{pmatrix} \alpha & t_r\\ 0 & \alpha \end{pmatrix}
\right)\in \GL(2,\CC)^{r+1} \mid\alpha\in\CC^*, t_i\in \CC \right\}
\]
\item If $\me=\ml_1\oplus\ml_2$ with $\ml_1\not\cong \ml_2$ and $\deg\ml_1\geq\deg\ml_2$, then
 $\Aut_B(\E)  \cong H_r '$, where
\[
H_r ':=\left\{\left(
\begin{pmatrix} \alpha & 0\\ 0 & \beta \end{pmatrix},\begin{pmatrix} \alpha & t_1\\ 0 & \beta \end{pmatrix},\cdots, \begin{pmatrix} \alpha & t_r\\ 0 & \beta \end{pmatrix}
\right)\in \GL(2,\CC)^{r+1} \mid\alpha,\beta\in\CC^*, t_i\in \CC \right\}
\]
where $r=h^0(B, \ml_1^{\otimes 2}\otimes(\det \me)^{-1})$.
\item If $\me=\ml\oplus \ml$, then $\Aut_B(\E)  \cong \GL(2,\CC)$.
\end{enumerate}
\end{enumerate}

\begin{cor}\label{cor: q geq 2}
Let $f\colon X=\PP(\E)\rightarrow B$ be a $\PP^1$-bundle over a smooth curve $B$ with $g(B)\geq 2$. Then  $\Aut_\QQ(X) = \Aut_\ZZ(X) = \Aut_B(X)$, $\Aut_0(X) \cong \Aut_B(\E)/\CC^*$ and $\Aut_\ZZ(X)/\Aut_0(X)\cong \Delta$, where $\Delta$ is as in (2) above.
\end{cor}
\begin{proof}
The automorphism group $\Aut_\QQ(X) = \Aut_\ZZ(X) $ induces a trivial action on $H^1(X, \ZZ) = H^1(B,\ZZ)$. Since $g(B)\geq 2$, we infer that $\Aut_\ZZ(X)$ induces a trivial action on $B$. Thus $\Aut_\ZZ(X)\subset \Aut_B(X)$. The inclusion in the other direction is clear. The isomorphism $\Aut_0(X) \cong \Aut_B(\E)/\CC^*$ and $\Aut_\ZZ(X)/\Aut_0(X)\cong \Delta$ come from (2) above.
\end{proof}

\begin{theo}\label{thm: q geq 2}
Let  $f\colon X\cong \PP(\E)\rightarrow B$ be a $\PP^1$-bundle over a curve of genus $ g(B)\geq 2$. 
Then the equalities $\Aut_\QQ(X) = \Aut_{\ZZ}(X) = \Aut_B(X)$ and $\Aut_\#(X)=\Aut_*(X)=\Aut_0(X)$ hold. 

Thus $\Gamma_\QQ(X) = \Gamma_\ZZ(X) \cong \Delta$, while $\Ga_\sharp(X)$ and $\Gamma_* (X)$ are trivial, where $\Gamma_\QQ(X)$, $\Gamma_\ZZ(X)$, $\Ga_\sharp(X)$ and $\Gamma_*(X)$ are the groups of connected components of $\Aut_\QQ(X)$, $\Aut_\ZZ(X)$, $\Aut_\#(X)$ and $\Aut_*(X)$ respectively.
\end{theo}
\begin{proof}
The first assertion is from Corollary~\ref{cor: q geq 2}, and we have a short exact sequence of groups
\[
1\rightarrow \Aut_0(X) \rightarrow \Aut_B (X) \rightarrow \Delta \rightarrow 1.
\]

We first consider the case where  $\deg \sE $ is even, so that we can assume that $\sE$ has zero degree, hence $\sE$ is a differentiably trivial vector bundle, and a fortiori  $X=\PP(\sE)$ is a differentiably trivial  $S^2$-bundle over $B$. 

That is, there is a diffeomorphism $\varphi\colon X\rightarrow B\times S^2$ such that the following diagram is commutative:
\[
\begin{tikzcd}
X \arrow[rr, "{\varphi}"]\arrow[rd, "f"'] & & B\times S^2 \arrow[ld, "\pr_1"] \\
 & B&
\end{tikzcd}
\]
where $\pr_1$ denotes the projection to the first factor.

We need to show that $\Aut_\sharp(X) = \Aut_0(X)$.

We 
have that $\Aut_{\QQ} (X) = \Aut_B(X)$ maps onto $\Delta$ with kernel $\Aut_0(X) $.
$\Aut_B(X)$ consists of sections $\s$ of the sheaf
$\PP GL(2, \hol_B) = \PP SL(2, \hol_B)$ and $\De$ measures the obstruction to lifting to a section of $ SL(2, \hol_B)$.

This obstruction is topological. $\Delta$ is a group of line bundles of 2 -torsion, hence it is a group of maps 
$ \delta(\sigma)\colon \pi_1(B)  \ra \ZZ/2\ZZ$. 

Since  $ X$ is differentiably trivial,  to $\s$ corresponds a diffeomorphism $$H  \colon B \times  \PP^1 \ra B \times  \PP^1,$$
linear on the fibres $\PP^1$, hence a 
 differentiable map
$s\colon B \ra \PP SL(2, \CC)$, and this map is liftable to $s'\colon B \ra  SL(2, \CC)$ if and only if 
$$\pi_1(s)\colon \pi_1(B)  \ra \pi_1( \PP SL(2, \CC)) =   \ZZ/2\ZZ$$
is trivial. This homomorphism is exactly $\de(\s)$, as it is easy to verify.

The diffeomorphism $H  \colon B \times  \PP^1 \ra B \times  \PP^1$,   via the second projection,  gives a continuous map 
$ h\colon B \ra \mathrm{ContMaps}  (\PP^1, \PP^1)_1$ of $B$ 
into  the space of continous selfmaps of degree 1 on $\PP^1$.

By a result of Graeme Segal, \cite{segal}, this space is homotopically equivalent  up to dimension $1$
to $\PP SL(2, \CC)$.

Hence, if a map $H$ is  homotopic to the identity, then also $h$ is homotopic to the identity, hence $h$  induces a trivial  
homomorphism of fundamental groups 
 $ \pi_1(B)  \ra \ZZ/2\ZZ$.

The conclusion is that 
 $\s \in \Aut_\#(X)$ maps trivially 
 to $\Delta$, hence $\s \in \Aut_0(X)$, see corollary \ref{cor: q geq 2}.
 
In  the case where  $\sE$ is of odd degree, take an \'etale double cover $B' \ra B$, so that the pull back 
$X'$ of $X$ is diffeomorphic to $\PP^1 \times B'$.

 We have that $\s \in \Aut_{\ZZ} (X)$ lifts to $X'$,
since $\s$ acts trivially on $H^1 (B, \ZZ/2\ZZ)$, and lifts to $\s' \in \Aut_{B'}(X')$. 

Observe that the pull-back map
$$H^1 (B, \ZZ/2\ZZ)  \ra H^1 (B', \ZZ/2\ZZ)$$ has kernel $\cong \ZZ/2$ generated by the class of the \'etale covering
$B' \ra B$. 
 
We have  that $\de(\s') $ is the pull-back of $\de(\s)$, and if $\de(\s) \neq 0 $ we can then choose $B' \ra B$ appropriately so that  $\de(\s') \neq 0$.

By the same token, if $\s$ is  homotopic to the identity, also $h' $ is  homotopic to the identity, hence to $\s$ 
corresponds the trivial  element in $\De$, and we are done also in this case argueing by contradiction.
\end{proof}

For completeness we summarize here some results by Maruyama.

\begin{thm}[{\cite[Theorem 3 and Remark 6]{Ma71}}]
Let $f\colon X=\PP(\E)\rightarrow B$ be a $\PP^1$-bundle over a smooth curve $B$ of genus  $q (X) \leq 1$. 

Then the following holds.
\begin{enumerate}[leftmargin=*]
\item Suppose that $q(X)=0$. If $X=\F_e$ with $e>0$, then $\Aut(X) =\Aut_0(X) $, and , more precisely, 
we have an exact sequence
\[
1\rightarrow \bar H_{e+1} \rightarrow\Aut(X) \rightarrow \PP\GL(2,\C) \rightarrow 1.
\]
where $\bar H_{e+1}:=H_{e+1}/\CC^*$; see the paragraph above corollary \ref{cor: q geq 2} for the definition of $H_{e+1}$.

\item Suppose that $B$ is elliptic and $X$ is decomposable. Then we have an exact sequence:
\[
1\rightarrow \Aut_B(X)\rightarrow \Aut(X) \rightarrow H \rightarrow 1
\]
where $H=\{\sigma\in\Aut(B)\mid \sigma^*(\ml^{\otimes 2}\otimes \det(\me)^{-1})\cong \ml^{\otimes 2}\otimes \det(\me)^{-1}\}$. Moreover, the following holds.
\begin{enumerate}[leftmargin=*]
\item If $e>0$ then $\Aut_B(X) =\Aut_0(X)\cong \bar H_r'$, where $r=h^0(\ml^{\otimes 2}\otimes \det(\me)^{-1})$ and $\bar H_r'=H_r'/\CC^*$ 
(see the paragraph above corollary \ref{cor: q geq 2} for the definition of $H_r'$.)

and we have an exact sequence
\[
1\rightarrow \Aut_B(X)\rightarrow \Aut_\ZZ(X) \rightarrow \Aut_0(B)[e] \rightarrow 1
\]
where $\Aut_0(B)[e]\cong(\ZZ/e\ZZ)^2$ denotes the $e$-torsion part of $\Aut_0(B)$.
\item If $e=0$ and $X$ has only one minimal section, then $\Aut_0(X) = \Aut_\ZZ(X)$, $\Aut_B(X) \cong \bar H_r'$, and we have an exact sequence
\[
1\rightarrow \Aut_B(X)\rightarrow \Aut_0(X) \rightarrow \Aut_0(B) \rightarrow 1
\]
\item If $e=0$ and $X$ has exactly two  minimal sections $C_1$ and $C_2$, then $\Aut_0(X)\cong X\setminus(C_1\cup C_2)$, where the latter is with the natural algebraic group structure, and we have $\Aut_B(X)=\Aut_0(X)\rtimes \ZZ/2\ZZ$, where $\ZZ/2\ZZ$ interchanges the two sections $C_1$ and $C_2$.
\item If $X=\PP^1\times B$ then $\Aut(X) = \Aut(\PP^1)\times \Aut(B) = \PP\GL(2, \CC)\times \Aut(B)$, and $\Aut_\ZZ(X)=\Aut_0(X) = \PP\GL(2, \CC)\times \Aut_0(B)$.
\end{enumerate}
\item Suppose that $B$ is elliptic and $X$ is indecomposable.
\begin{enumerate}
\item If $e=0$ then we have an exact sequence
\[
1\rightarrow \CC^* \rightarrow \Aut(X) \rightarrow \Aut(B) \rightarrow 1.
\]
In this case, $\Aut_\ZZ(X) = \Aut_0(X)\cong X\setminus C$ where $C\subset X$ is the unique minimal section, it is a nontrivial extension of $\Aut_0(B)$ by $\CC^*$. 
\item If $e=1$ then we have an exact sequence
\[
1\rightarrow \Delta \rightarrow \Aut(X) \rightarrow \Aut(B) \rightarrow 1.
\]
where $\Delta=\Pic^0(B)[2]\cong(\ZZ/2\ZZ)^2$. Furthermore, $\Aut_\ZZ(X)=\Aut_0(X)$ and $\Delta$ is contained in $\Aut_0(X)$, so there is an exact sequence
\[
1\rightarrow \Delta \rightarrow \Aut_0(X) \rightarrow \Aut_0(B) \rightarrow 1.
\]

\end{enumerate}
\end{enumerate}
\end{thm}

\subsection{Case $\kappa(X)=0$}
 In this section, we treat the surfaces $X$ with $\kappa(X)=0$. 
The following is a list of known facts:

\begin{enumerate}[leftmargin=*]
\item If $X$ is K3 surface, then $\Aut_\QQ(X)=\{\id_X\}$ by \cite{br}.

\item If $X$ is an Enriques surface, then $|\Aut_\QQ(X)|\leq 4$ and $|\Aut_\ZZ(X)|\leq 2$, and both bounds are sharp \cite{MN84}. The fact that $\Aut_\ZZ(X)$ can be nontrivial for an Enriques surface contradicts the last statement of \cite[Theorem~2.2]{Pe80}.

\item If $X$ is an Abelian surface, then $\Aut_{\QQ}(X) = \Aut_0(X)\cong X$.
\end{enumerate}

Suppose now that $X$ is a hyperelliptic surface  (bielliptic surface in the notation of \cite{Be96}). These are the prototype examples
of an SIP of unmixed type, and were classified by Bagnera and de Franchis, resp.~ by Enriques-Severi (\cite{bdf}, \cite{es9}, \cite{es10}).
 These are $X=(F \times E)/\De_G$ with $E$ and $F$ elliptic curves and $G$ acting freely on $E$, while  $g(F/G)=0$. 
 
 We can apply 
 Principle  \ref{SIPU} 
 and obtain the following theorem.

\begin{theo}\label{thm: hyperelliptic}
Let $X=(F \times E)/\Delta_G$ be an hyperelliptic  surface in the  above notation.

 Then  $\Aut_\ZZ(X) \cong 
 E
 = \Aut_0(X)$
and $\Aut_\QQ(X)/\Aut_\ZZ(X)$ is isomorphic to one of following groups:
\[
1,\, \ZZ/2\ZZ,\, ( \ZZ/2\ZZ)^2,\,  \mathfrak S_3,\, 
D_4, \,
\mathfrak A_4
\]
where $\mathfrak S_3$ denotes the symmetric group on three elements, $D_4$ is the dihedral group of order $8$, and $\mathfrak A_4$ is the alternating group on $4$ elements. 

Moreover, $\Aut_\QQ(X)/\Aut_\ZZ(X)\cong \mathfrak A_4$ if and only if $F=F_\omega$ and $G\cong\ZZ/2\ZZ$. In particular, $|\Aut_\QQ(X)/\Aut_\ZZ(X)|
\leq 12$ and the equality is attained if and only if $F=F_\omega$, the equianharmonic (Fermat) elliptic curve,  and $G\cong\ZZ/2\ZZ$ ($F_\omega = \CC/(\ZZ\oplus \ZZ\omega)$ with $\omega$ a primitive 3rd root of unity).
\end{theo}

\begin{proof}

By Principle \ref{SIPU}, (I) and (III),  $\Aut_{\QQ}(X) \subset N_{\De_G} / \De_G$ corresponds to the automorphisms $h = (h_1, h_2)$
such that $h_2$ is a translation, and $h_1 $ acts trivially on $H^1 (F, \ZZ)^G$. 

While $\Aut_{\ZZ}(X) \subset N_{\De_G} / G$
corresponds to the subgroup  $$ G \times E < \Aut(F) \times \Aut(E).$$

 Since $H^1(F/G, \ZZ) = 0$, $\Aut_{\QQ}(X)/\Aut_\ZZ(X) \cong N_G / G$, where 
$N_G$ is the  normalizer of $G$ in $\Aut(F)$.

From the  identification
\[
\Aut_{\QQ}(X)/\Aut_\ZZ(X) \cong \{\gamma\in \Aut(F)\mid \gamma G=G\gamma\}/G = N_G/G.
\]

@e proceed as follows to determine $N_G/G$: consider the split short exact sequence
\[
0\rightarrow \Aut_0(F) 
\cong  F \rightarrow \Aut(F) \xrightarrow{\varphi} A \rightarrow 0
\]
where $A\subset \Aut(F)$ is the subgroup preserving the group structure of $F$. 

For convenience of notation we write $F$ instead of $ \Aut_0(F) $. Restricting to $G$ and $N_G$ we obtain short exact sequences
\[
0\rightarrow G\cap F \rightarrow G \rightarrow \varphi(G)\rightarrow 1
\]
and 
\[
0\rightarrow N_G\cap F \rightarrow N_G \rightarrow \varphi(N_G) \rightarrow 1.
\]
Therefore, we have a short exact sequence
\[
0\ra \left(N_G\cap F \right)/\left(G\cap F\right) \rightarrow N_G/G \rightarrow \varphi(N_G)/\varphi(G) \rightarrow 1.
\]

Next we divide the discussion into cases according to \cite[List VI.20]{Be96}. 
\begin{enumerate}[leftmargin=*]
    \item[(1)] $G\cong\ZZ/2\ZZ$ acting on $F$ by $x\mapsto -x$. 
    In this case, $N_G\cap F \cong(\ZZ/2\ZZ)^2$ consists of translations by $2$-torsion points, and 
    \[
    \varphi(N_G)=
    \begin{cases}
    \langle x\mapsto -x \rangle \cong \ZZ/2\ZZ & \text{ if } F\neq F_i, F_\omega\\
    \langle x\mapsto ix \rangle \cong\ZZ/4\ZZ &\text{ if } F= F_i \\
     \langle x\mapsto -\omega x \rangle \cong \ZZ/6\ZZ & \text{ if } F=F_\omega
    \end{cases}
    \]
    For $G$ we have 
    \[
    G\cap F  =\{ 0\}\text{ and }\varphi(G) =G=\langle x\mapsto -x\rangle \cong\ZZ/2\ZZ
    \]
We have thus a split short exact sequence
\[
    0\rightarrow N_G\cap F \rightarrow N_G/G \rightarrow \varphi(N_G)/G \rightarrow 0
    \]
and it is now easy to determine  the corresponding semidirect product
\[
N_G/G \cong 
  \begin{cases}
   (\ZZ/2\ZZ)^2 & \text{ if } F\neq F_i, F_\omega\\
   D_4 &\text{ if } F= F_i \\
    \mathfrak  A_4 & \text{ if } F=F_\omega
    \end{cases}
\]
where $D_4$ denotes the dihedral group of order $8$ and $\mathfrak A_4$ is the alternating group on $4$ elements.
    \item[(2)] $G\cong \ZZ/2\ZZ\oplus \ZZ/2\ZZ$ acting on $F$ by $x\mapsto -x, x\mapsto x+\epsilon$ with $\epsilon$ a 
    nontrivial 2-torsion point of $F$.
     In this case, $N_G\cap F \cong(\ZZ/2\ZZ)^2$ consists of the $2$-torsion points, and 
 $\varphi(N_G) = \varphi(G) = \langle x\mapsto -x\rangle$ if $ F \neq F_i$, while $\varphi(N_G) = A$
     if $F = F_i$ and $ \epsilon  = \frac{1}{2} ( 1 + i)$. 
     
     Thus $N_G/G\cong\ZZ/2\ZZ$ or it may be $\ZZ/2\ZZ \oplus \ZZ/2\ZZ$ for $F = F_i$.
    \item[(3)] $G\cong \ZZ/4\ZZ$ acting on $F=F_i = \CC/(\ZZ\oplus i\ZZ)$ by $x\mapsto i x$. In this case, $N_G\cap F = \langle x\mapsto x+\frac{1+i}{2}\rangle\cong\ZZ/2\ZZ$ and $\varphi(N_G) = \varphi(G)=\langle x\mapsto ix\rangle$. 
    
    Thus $N_G/G\cong\ZZ/2\ZZ$.
    \item[(4)] $G\cong\ZZ/4\ZZ \oplus \ZZ/2\ZZ$, acting by $x\mapsto ix, x\mapsto x+\left(\frac{1+i}{2}\right)$. In this case, $N_G\cap F\cong(\ZZ/2\ZZ)^2$ is the group of  $2$-torsion points of $F$, and $\varphi(N_G) = \varphi(G)=\langle x\mapsto ix\rangle$. Thus $N_G/G\cong\ZZ/2\ZZ$.
    \item[(5)] $G\cong\ZZ/3\ZZ$ acting on $F=F_\omega$ by $x\mapsto \omega x$. In this case, $N_G\cap F=\langle x\mapsto x+\frac{1-\omega}{3}\rangle\cong\ZZ/3\ZZ$, and $\varphi(N_G) = \langle x\mapsto -\omega x\rangle\cong \ZZ/6\ZZ$.
    
     It follows that $N_G/G\cong \mathfrak S_3$, the symmetric group on three elements.
    \item[(6)] $G\cong \ZZ/3\ZZ\oplus \ZZ/3\ZZ$ acting by $x\mapsto \omega x, x\mapsto x+\left(\frac{1-\omega}{3}\right)$. In this case, $N_G\cap F \cong(\ZZ/3\ZZ)^2$ is the group of  $3$-torsion points of $F$, and $\varphi(N_G)=\langle x\mapsto -\omega x\rangle\cong \ZZ/6\ZZ$. It follows that $N_G/G\cong \mathfrak S_3$.
    \item[(7)] $F=F_\omega$ and $G\cong \ZZ/6\ZZ$ acting by $x\mapsto -\omega x$. In this case, $N_G\cap F=\{ 0 \}$ and $\varphi(N_G)=\varphi(G)=G=\langle x\mapsto -\omega x\rangle$. It follows that $N_G/G$ is trivial.
\end{enumerate}
\end{proof}

\begin{thm}\label{thm: k=0a}
Let $X$ be a smooth projective surface with $\kappa(X)=0$. Then $[\Aut_\QQ(X):\Aut_0(X)] \leq 12$.
\end{thm}
\begin{proof}
If $X$ is minimal, the assertion follows from the listed facts above and Theorem~\ref{thm: hyperelliptic}. 

If $X$ is not minimal, then $\Aut_\QQ(X)\subset \Aut_\QQ(X_{\min})$ by Principle~\ref{prin: descend}. So, if $\dim \Aut_0(X_{\min})=0$, then
\[
|\Aut_\QQ(X)|\leq |\Aut_\QQ(X_{\min})| \leq 12.
\]
In case $\dim \Aut_0(X_{\min})>0$, $X_{\min}$ is either an abelian surface or a hyperelliptic surface and $\Aut_0(X_{\min})$ is an abelian variety of dimension $2$ or $1$. The subgroup $\Aut_\QQ(X)\subset \Aut_\QQ(X_{\min})$ fixes the points $p\in X_{\min}$, over which $X\rightarrow X_{\min}$ is not an isomorphism. But an element of $\Aut_0(X_{\min})$ fixes a point if and only if it is the identity. It follows that 
\[
\Aut_\QQ(X) \cap \Aut_0(X_{\min})=\id_{X_{\min}},
\]
and hence the natural homomorphism $\Aut_\QQ(X)\rightarrow  \Aut_\QQ(X_{\min})/\Aut_0(X_{\min})$ is injective. Therefore,
\[
|\Aut_\QQ(X)|\leq |\Aut_\QQ(X_{\min})/\Aut_0(X_{\min})|\leq 12.
\]
\end{proof}
For $\Aut_\sharp(X)$ we have a sharper result:

\begin{thm}\label{thm: k=0b}
Let $X$ be a smooth projective surface with $\kappa(X)=0$. Then $\Aut_\sharp(X) = \Aut_0(X)$.
\end{thm}
\begin{proof}
Let $X_{\min}$ be the minimal model of $X$. Then by 
 Principle~\ref{prin: descend},
$\Aut_\ZZ(X)\subset \Aut_\ZZ(X_{\min})$.

If $X_{\min}$ is a K3 surface, then $\Aut_\ZZ(X_{\min})$ is trivial. It follows that $\Aut_\ZZ(X)$ and hence $\Aut_\sharp(X)$ is trivial. 

If $X_{\min}$ is an Enriques surface, then its universal cover $\tilde X_{\min}$ is a K3 surface and is the minimal model of the universal cover 
 $\pi\colon \tilde X\rightarrow X$.
 Suppose that $\sigma\in\Aut_\sharp(X)$ and that $\Sigma\colon X\times I\rightarrow X$ is a homotopy from  $\id_X$ to $\sigma$. Then, by the homotopy lifting property, there is a homotopy $\tilde\Sigma\colon \tilde X\times 
  [0,1]
 \rightarrow \tilde X$ from $\id_{\tilde X}$ to  $\tilde \sigma$, where $\tilde \sigma\in\Aut(\tilde X)$ is a lifting of $\sigma$:
 
\[
\begin{tikzcd}
\tilde X\times [0,1]\arrow[d, "\pi\times\id"] \arrow[rr, "\exists\,\, \tilde \Sigma"]&& \tilde X  \arrow[d,"\pi"]\\
X\times [0,1]\arrow[rr, "\Sigma"] && X
\end{tikzcd}
\]
 Since $\Aut_\sharp(\tilde X)$ is trivial, $\tilde \sigma =\id_{\tilde X}$. It follows that $\sigma=\id_X$. 

If $X=X_{\min}$ is an Abelian surface or a hyperelliptic surface, then  by 
 Principle~\ref{SIPU}
 and the discussion above we know that $\Aut_\ZZ(X) = \Aut_0(X)$. It follows a  fortiori  that $\Aut_\sharp(X) = \Aut_0(X)$. 

Finally,  suppose that $X_{\min}$ is an Abelian surface or a hyperelliptic surface but $\rho\colon X\rightarrow X_{\min}$ is not an isomorphism. Then the topological Euler characteristics satisfy $\chi_\topo(X)>\chi_\topo(X_{\min})=0$. 

Now observe that there is a flat metric on $X_{\min}$, and by Principle~\ref{prin: rigidity}, we have $\Aut_\sharp(X) =\{ \id_X\}$.
\end{proof}

\subsection{Case $\kappa(X)=2$}
For a surface $X$ of general type, we have that $|\Aut_\QQ(X)|$ is bounded and, in fact, $|\Aut_\QQ(X)|\leq 4$ if $\chi(\mo_X)\geq189$ (\cite{Cai04}). Subsequently, examples have been  found with $|\Aut_\QQ(X)|= 4$ and $\chi(\mo_X)$ arbitrarily large (\cite{CL18}).

\end{document}